\documentclass[11pt]{amsart}

\usepackage{amsmath, amsthm, amssymb, amsfonts, enumerate}
\usepackage[colorlinks=true,linkcolor=blue,urlcolor=blue]{hyperref}
\usepackage{dsfont}
\usepackage{color}
\usepackage{geometry}
\usepackage{todonotes}

\newtheorem{theorem}{Theorem}[section]
\newtheorem{remark}[theorem]{Remark}
\newtheorem{assumption}[theorem]{Assumption}
\newtheorem{lemma}[theorem]{Lemma}
\newtheorem{proposition}[theorem]{Proposition}

\newtheorem{definition}[theorem]{Definition}

\def \E{\mathsf{E}}

\def \P{\mathsf{P}}
\def \R{\mathbb{R}}

\def \Q{\mathsf{Q}}

\def\d{\mathrm{d}}

\DeclareMathOperator*{\esssup}{ess\,sup}
\DeclareMathOperator*{\essinf}{ess\,inf}

\title[Irreversible Investment under Knightian Uncertainty]{A Knightian Irreversible Investment Problem}
\author[Ferrari]{Giorgio Ferrari}
\author[Li]{Hanwu Li}
\author[Riedel]{Frank Riedel}

\address{G.~Ferrari: Center for Mathematical Economics (IMW), Bielefeld University, Universit\"atsstrasse 25, 33615, Bielefeld, Germany}
\email{\href{mailto:giorgio.ferrari@uni-bielefeld.de}{giorgio.ferrari@uni-bielefeld.de}}
\address{H.\ Li: Center for Mathematical Economics (IMW), Bielefeld University, Universit\"atsstrasse 25, 33615, Bielefeld, Germany}
\email{\href{mailto:hanwu.li@uni-bielefeld.de}{hanwu.li@uni-bielefeld.de}}
\address{F.\ Riedel: Center for Mathematical Economics (IMW), Bielefeld University, Universit\"atsstrasse 25, 33615, Bielefeld, Germany}
\email{\href{mailto:frank.riedel@uni-bielefeld.de}{frank.riedel@uni-bielefeld.de}}

\date{\today}

\numberwithin{equation}{section}

\begin{document}

\begin{abstract}
In this paper, we study an irreversible investment problem under Knightian uncertainty. In a general framework, in which Knightian uncertainty is modeled through a set of multiple priors, we prove existence and uniqueness of the optimal investment plan, and derive necessary and sufficient conditions for optimality. This allows us to construct the optimal policy in terms of the solution to a stochastic backward equation under the worst-case scenario. In a time-homogeneous setting -- where risk is driven by a geometric Brownian motion and Knightian uncertainty is realized through a so-called ``$\kappa$-ignorance'' -- we are able to provide the explicit form of the optimal irreversible investment plan.
\end{abstract}

\maketitle

\smallskip


\smallskip

{\textbf{MSC2010 subject classification}}: 93E20, 91B38, 65C30.

\smallskip


\section{Introduction}

The theory of irreversible investment under uncertainty has received much attention in Economics as well as in Mathematics (see, for example, the extensive review in Dixit and Pindyck \cite{DP}). Kenneth  Arrow pioneered the analysis  in \cite{A} by solving  a deterministic irreversible investment problem, while \cite{B98} and \cite{P88} extend the analysis to a stochastic setting. There, randomness enters the model through a geometric Brownian motion representing an exogenous economic shock, and the firm's profit function is of Cobb-Douglas type. Afterward, several other modeling features have been introduced, such as the presence of jumps in the dynamics of the economic shock (e.g., \cite{B04, BL}), or regime switching models \cite{GMM}. From a mathematical point of view, optimal irreversible investment problems under uncertainty can be modeled as singular stochastic control problems; in fact, the economic constraint that does not allow disinvestment is well encompassed by thinking of the cumulative investment as a monotone control process. Singular stochastic control problems with applications to irreversible investment have been addressed via different methods ranging from the dynamic programming approach (\cite{Kobila, LoZe, Pham}, among many others), to the theory of $r$-excessive mappings \cite{Alvarez1, Alvarez2}, the connection to optimal stopping (e.g., \cite{CH2, DeAFeFe}), and stochastic representation problems \cite{BaBe, CF, Fe15, FeSa}.

Recently, Riedel and Su investigated in \cite{RS} an irreversible investment problem in a general not necessarily Markovian setting. They proved existence and uniqueness of the optimal investment plan, and constructed it as the minimal investment needed to keep the production capacity above the so-called ``base capacity''. This is an endogenously determined minimal level of production, and it is mathematically characterized as the solution to a backward equation deriving from first-order conditions for optimality. The kind of backward equation considered in \cite{RS} appeared for the first time in \cite{BR} (in the context of an optimal consumption problem under intertemporal substitution), and its general mathematical analysis was later developed in \cite{BE}.

All the above works are based on the assumption that, although the firm does not know the realization of the future profits, their probability distribution is perfectly known. In this paper we study an irreversible investment problem under Knightian uncertainty (or ambiguity), where the latter is modeled through a set of multiple priors for the profits' distribution. In particular, we assume that the firm is ambiguity-averse and uses a so-called $g$-expectation in order to evaluate its net profits. The $g$-expectation, introduced by Peng in \cite{P97}, characterizes the nonlinear expectation of a random variable $X$ in terms of (the first component of) the solution to a backward stochastic differential equation (BSDE) with terminal value $X$ and driver $g$. The advantage of $g$-expectation is that the Knightian uncertainty is fully characterized by the real-valued function $g$. In order to reflect the firm's ambiguity aversion we assume that the driver $g$ is concave and, within this setting, we let the firm maximize its $g$-expected profits, net of the total proportional costs of investment. In our formulation the resulting optimization problem thus takes the form of a \emph{singular stochastic control problem under Knightian uncertainty}. Indeed, the firm increases the level of the production capacity not necessarily via investment rates; also lump-sum or singularly-continuous interventions are allowed. To the best of our knowledge, this is the first work addressing such a problem.

Our first result establishes existence and uniqueness of the optimal investment plan. As in \cite{RS}, this is accomplished by a suitable application of Koml\'{o}s theorem (in the version of \cite{K99}), upon proving that any feasible investment plan is not larger than the one which is optimal if one neglects the irreversibility constraint.

We then move on by proving necessary and sufficient first-order conditions for optimality. Since we consider the irreversible investment problem under a set of multiple priors $\mathcal{P}$, those conditions involve a subclass of probability measures under which the net expected profits are minimal; these are the so-called worst-case scenarios. Roughly speaking, if the capacity $C^{\star}$ induced by the investment plan $I^{\star}$ satisfies the first-order conditions under a worst-case scenario $\P^\star$, then (under suitable conditions) $I^{\star}$ is the optimal investment policy and viceversa. The first-order conditions naturally degenerate into those in \cite{RS} when $\mathcal{P}$ is a singleton. It is worth pointing out that all the above results, such as existence, uniqueness, and first-order conditions, are beyond the Markovian framework. Moreover, their proofs needed refined technical arguments employing results from BSDEs' theory, and as such they are a non trivial generalization of those in \cite{RS}.

Although the previous existence and uniqueness result has a clear mathematical value, it does not provide the structure of the optimal investment plan. Inspired by \cite{RS}, we then construct the optimal investment plan as a ``base capacity policy''; i.e., in terms of a base capacity process solving a suitable stochastic backward equation deriving from the first-order conditions for optimality. However, differently to \cite{RS}, in our problem the optimal base capacity has to be determined together with the probability measure realizing the worst-case scenario for the net profits that are accrued via the optimal base capacity policy. This clearly makes our problem much more involved than the one without Knightian uncertainty.


In order to obtain an explicit expression for the optimal base capacity and its associated optimal investment plan, we specialize our setting to a time-homogeneous one where: the time horizon is infinite, the economic shock driving the profits is a geometric Brownian motion, the profit function is of Cobb-Douglas type, and the nonlinear expectation is the so-called ``$\kappa$-ignorance" (denoted by $\mathcal{E}^{-\kappa}[\,\cdot\,]$). This is one of the most important examples of $g$-expectation and can be obtained by taking the driver $g$ of the form $g(t,z)=-\kappa|z|$. In this case, the $g$-expectation is a lower expectation, which means that there exists a set of probability measures $\mathcal{P}_\kappa$, such that $\mathcal{E}^{-\kappa}[X]=\inf_{\P\in\mathcal{P}_\kappa}\E^\P[X]$, for any square-integrable random variable $X$. More precisely, $\mathcal{P}_\kappa$ is the collection of all probability measures $\P^\xi$, that are equivalent to a reference one and with Girsanov kernel $\xi$ which is bounded by $\kappa$. The advantage of $\kappa$-ignorance is that the single parameter $\kappa$ completely characterizes the degree of ambiguity; i.e., the more ambiguity, the larger $\kappa$. Within such a setting, we show via a probabilistic verification theorem that the worst-case scenario is $\P^{-\kappa}$, under which the economic shock has the lowest drift, and that the optimal investment plan is the one which is optimal under $\P^{-\kappa}$. By borrowing ideas from \cite{RS}, the latter is explicitly constructed by solving the stochastic backward equation for the base capacity under $\P^{-\kappa}$. With the concrete form of optimal investment plan at hand, we can also show that: (i) the optimal investment is decreasing with respect to the interest rate $r$; (ii) the optimal investment is increasing in the economic shock; (iii) the firm would like to decrease the scale of investment as it faces more ambiguity.

It is worth pointing out that Nishimura and Ozaki \cite{NO} and Thijssen \cite{T} also study an irreversible investment problem under Knightian uncertainty. However, their setting is different from ours. In \cite{NO}, it is assumed that there is ambiguity over certain characteristics regarding the operating profits, which are characterized by a geometric Brownian motion, and the decision of the firm is to choose the best entry time into the market so as to maximize its net profits. On the other hand, in \cite{T} the ambiguity is about the discount rate. In fact, both \cite{NO} and \cite{T} formulate the problem as an optimal stopping one (see also Section 4.2 in \cite{CR}), rather than as a singular stochastic control problem like ours.

The rest of the paper is organized as follows. In Section \ref{sec:formulation} we present the setting and formulate the irreversible investment problem. The existence and uniqueness result is then proved in Section \ref{sec:existence}. Section \ref{sec:FOCs} presents the first-order conditions and then gives the characterization of the optimal investment plan in terms of the base capacity. Some properties of the optimal policy are also discussed in this section. In Section \ref{sec:casestudy} we finally provide the explicit solution to the irreversible investment problem in a time-homogeneous diffusive setting where the Knightian uncertainty is modeled by ``$\kappa$-ignorance". Appendix \ref{Appendix} contains an overview on $g$-expectation, while technical results are collected in Appendices \ref{AppendixC} and \ref{AppendixB}.

	
	\section{Setting and Problem Formulation}
	\label{sec:formulation}
	
	In this section, we introduce the irreversible investment problem under Knightian uncertainty which is the object of our study.
	
	Given a fixed time interval $[0,T]$, we let $(\Omega,\mathcal{F},\mathbb{F}:=\{\mathcal{F}_t\}_{t\in[0,T]}, \P_0)$ be a complete filtered probability space with filtration satisfying the usual conditions. Throughout this paper, $\E$ will be denoting the expectation under $\P_0$. Let $X:=\{X_t\}_{t\in[0,T]}$ be an $\mathbb{F}$-progressively measurable process taking values in some Banach space $E$. Here, $X$ can be regarded as an exogenous economic shock, e.g.\ the demand of a produced good, the state of technological improvement, or macroeconomic conditions.
	
Due to the irreversibility constraint, the firm chooses an investment plan $I=\{I_t\}_{t\in[0,T]}$ which is a nondecreasing, right-continuous and $\mathbb{F}$-adapted process such that $I_{0^-}=0$ a.s. The production capacity $C^I:=\{C^I_t\}_{t\in[0,T]}$ associated to an investment plan $I$ evolves as
	\begin{equation}
	\label{i2}
	\d C_t^I=-\delta C_t^I \d t + \d I_t, \ \  C^I_{0^-}=c \geq 0,
	\end{equation}
	where $\delta\geq 0$ is a given depreciation rate.
	Notice that, by the method of variation of constants, we can write
	\begin{equation}
	\label{solC}
	C_t^I = e^{-\delta t}\Big[c + \int_0^t e^{\delta s} \d I_s\Big], \quad t\geq 0.
	\end{equation}
	Here, and in the following, we interpret the integrals with respect to the (random) Borel-measure $\d I$ on $[0,T]$ in the Lebesgue-Stieltjes sense as $\int_0^T (\,\cdot\,) \d I_t:= \int_{[0,T]} (\,\cdot\,) \d I_t$. In such a way, a possible initial jump of the process $I$ (i.e.\ an initial lump sum investment) is taken into account in the integral.

When the economic shock at time $t$ is $X_t$ and the level of production capacity is $C_t$, the instantaneous operating profit of the firm is $\pi(X_t,C_t)$. Here, $\pi:E\times\mathbb{R}_+ \to \mathbb{R}_+$ is a continuous function satisfying the following \underline{standing assumption}.
	\begin{assumption}
	\label{a1}
	\hspace{10cm}
		\begin{itemize}
			\item[(1)] For each fixed $x\in E$, $\pi(x,\cdot)$ is strictly increasing, strictly concave and continuously differentiable;
			\item[(2)] For each fixed $x\in E$, the partial derivative $\pi_c(x,c)$ satisfies the so-called Inada conditions; that is,
			\begin{displaymath}
				\lim_{c \downarrow 0}\pi_c(x,c)=+\infty, \quad \textrm{ and } \quad \lim_{c \uparrow \infty}\pi_c(x,c)=0;
			\end{displaymath}
			\item[(3)] For each fixed $x\in E$, $\pi(x,0)=0$.
		\end{itemize}
	\end{assumption}

		Increasing the production capacity, the firm incurs proportional costs. We assume that the price of capital goods used to increase capacity is taken as num\'eraire so that the marginal cost of investment is normalized to one. Also, the firm discounts profits and costs at a constant interest rate $r>0$ (now expressed in terms of capital goods, not money). Hence, the cumulative profits of the company, net of the investment costs, are
$$\displaystyle \int_0^T e^{-rt}\pi (X_t,C_t^I) \d t - \int_0^T e^{-rt} \d I_t.$$

We assume that the firm is ambiguity-averse, and that the firm's evaluation of an uncertain gain or profit $\xi$ is given by a functional $\mathcal{E}^g[\xi]$, where $\mathcal{E}^g[\,\cdot\,]$ is a so-called $g$-expectation, whose generator $g:[0,T]\times\Omega\times\mathbb{R}\to\mathbb{R}$ satisfies the following assumption (we refer to Appendix \ref{Appendix} for the definition and important properties of $g$-expectation).
    \begin{assumption}
		\label{a2}
		\hspace{10cm}
    	\begin{itemize}
    		\item[(i)] for any $z\in \mathbb{R}$, $(t,\omega) \mapsto g(t,\omega,z)$ is $\mathbb{F}$-progressively measurable and
   $$\displaystyle \E\bigg[\int_0^T |g(t,\omega, z)|^2 \d t\bigg]<\infty;$$
    		
    		\item[(ii)] There exists a constant $\kappa>0$ such that, for all $(t,\omega) \in [0,T] \times \Omega$ and $z,z'\in \R$,
    		$$\displaystyle |g(t,\omega,z)-g(t,\omega,z')|\leq \kappa|z-z'|;$$
				
    		\item[(iii)] For any $(t,\omega) \in [0,T] \times \Omega$, $z \mapsto g(t,\omega,z)$ is concave;
				
    		\item[(iv)] For any $(t,\omega) \in [0,T] \times \Omega$, $g(t,\omega,0)=0$.
\end{itemize}
\end{assumption}

We now recall here a representation for $g$-expectation (see \cite{EPQ}) that will be useful in most of our subsequent analysis.
\begin{proposition}
\label{prop:A1}
Let the function $g$ satisfy Assumption \ref{a2} (i)-(iii) and denote by $D:=D_g$ the collection of all progressively measurable processes $\{\theta_t\}_{t \in [0,T]}$ such that
$$\E\bigg[\int_0^T |f(s,\theta_s)|^2 \d s\bigg]<\infty,$$
with $f(t,\omega,\theta)=\sup_{z\in\mathbb{R}}(g(t,\omega,z)-z \theta)$ being the convex dual of $g$.

For any $0\leq t \leq s\leq T$ and any $\mathcal{F}_s$-measurable and square-integrable random variable $\xi$, we have the following representation for the $g$-conditional expectation at time $t$
	\begin{equation}
	\label{repr-g}
	\mathcal{E}_{t,s}^g[\xi]=\essinf_{\theta\in D_g}\big\{\E^{\P^\theta}_t[\xi]+\alpha_{t,s}(\theta)\big\},
	\end{equation}
	where the penalty function is
	\begin{displaymath}
	\alpha_{t,s}(\theta):=\E^{\P^\theta}_t\bigg[\int_t^s f(r,\theta_r) \d r\bigg],
	\end{displaymath}
	and where $\E^{\P^\theta}_t[\,\cdot\,]$ is the conditional expectation taken under the probability measure $\P^\theta$ such that
	\begin{displaymath}
	\frac{\d\P^\theta}{\d\P_0}\Big|_{\mathcal{F}_t}=\exp\Big(\int_0^t \theta_s \d B_s-\frac{1}{2}\int_0^t \theta_s^2 \d s\Big), \quad 0 \leq t \leq T.
	\end{displaymath}
\end{proposition}

On the other hand, because of the ambiguity-aversion, the evaluation of an uncertain cost $\eta$ is given by $\mathcal{E}^{\widetilde{g}}[\eta]$, where $\widetilde{g}(t,z):=-g(t,-z)$. It is easy to check that $\mathcal{E}^{\widetilde{g}}[\eta]=-\mathcal{E}^g[-\eta]$. For example, if $g(t,z)=-\kappa|z|$, we set $\mathcal{E}^g[\xi]:=\mathcal{E}^{-\kappa}[\xi]$ and $\mathcal{E}^{\widetilde{g}}[\xi]:=\mathcal{E}^\kappa[\xi]$. In fact, by Proposition \ref{prop:A1}, if the function $g$ satisfies (i)-(iii), the $g$-expectation corresponds to a variational preference. Furthermore, if the function $g$ also satisfies (iv), then the $g$-expectation fulfills almost all the properties of the classical expectation, such as translation invariance, local property, with the exception of linearity (see Appendix \ref{Appendix} for details).

\begin{remark}\label{re}
Suppose that $g$ satisfies  Assumption \ref{a2}.  By the comparison theorem for BSDEs, for any $\mathcal{F}_T$-measurable and square integrable random variable $\xi$, we have
$$\mathcal{E}^g[\xi]\leq \mathcal{E}^\kappa[\xi], \ \ \mathcal{E}^{\widetilde{g}}[\xi]\leq \mathcal{E}^\kappa[\xi].$$
\end{remark}

Denoting by $\mathcal{A}_g$ the collection of all the investment plans such that
\begin{equation}
\label{eq:constraint}
\mathcal{E}^{\widetilde{g}}\bigg[\int_0^T e^{-rt} \d I_t\bigg]<\infty,
\end{equation}
and defining, for any $I \in \mathcal{A}_g$,
\begin{equation}
	\label{ii}
	\Pi(I):=\mathcal{E}^{g}\bigg[\int_0^T e^{-rt}\pi (X_t,C_t^I) \d t- \int_0^T e^{-rt} \d I_t\bigg],
	\end{equation}
the firm aims at solving the problem
	\begin{equation}
	\label{OC}
	V^{\star}:=\sup_{I \in \mathcal{A}_g}\Pi(I)
	\end{equation}
	In the following, any investment plan belonging to $\mathcal{A}_g$ will be called \emph{admissible}. Notice that, under \eqref{eq:constraint}, $V^{\star}$ is well defined, even if possibly infinite. It is also worth stressing that no Markovian assumption has been made.
	

\section{Existence and Uniqueness of the Optimal Investment Plan}
\label{sec:existence}

In this section we show that there exists a unique irreversible investment plan that is optimal for problem \eqref{OC}. In order to accomplish that, we first need to impose some integrability condition on the underlying stochastic process $X$. For this purpose, for each fixed $x \in E$, $r>0$ and $\delta\geq 0$, define
\begin{equation}
\label{pistar}
\pi^{\star}(x,r,\delta):=\sup_{c\geq 0}\big(\pi(x,c)-(r+\delta)c\big),
\end{equation}
and denote by $c^{\star}(x,r,\delta)$ the unique maximizer of $c\mapsto \pi(x,c)-(r+\delta)c$ which, by Assumption \ref{a1}, is completely characterized by the first-order condition
\begin{equation}
\label{cstar}
\pi_c\big(x,c^{\star}(x,r,\delta)\big)=r+\delta.
\end{equation}
Then we make the following assumption.
	\begin{assumption}
	\label{a3}
	\hspace{10cm}
	 	\begin{itemize}
	 	\item[(H1)] For any $t\in[0,T]$, $\E\big[|\pi^{\star}\big(X_t,r,\delta\big)|^2\big]<\infty$;
		
	 	\item[(H2)] $\E\big[\sup_{s\in[0,T]}|c^{\star}(X_s,r,\delta)|^2\big]<\infty$.
	 \end{itemize}
	\end{assumption}

\begin{remark}\label{remark3.2}
The auxiliary function $\pi^{\star}$ defines the maximal profit that the company can achieve if the investment plan were perfectly reversible; in fact, in such a case, the marginal operating profit $\pi_c$ equates to the user cost of capital $r+\delta$.

If the operating profit function is of the form $\pi(x,c)=e^x \frac{c^{1-\alpha}}{1-\alpha}$, with $\alpha\in(0,1)$, and if $X$ is a L\'{e}vy process with bounded positive jumps, then Assumption \ref{a3} holds true (see Example 2.2 in \cite{RS}).
\end{remark}
	
We can now state the existence and uniqueness result for a solution to problem \eqref{OC}.
	\begin{theorem}
	\label{ii1}
	Under Assumptions \ref{a1}, \ref{a2} and \ref{a3}, there exists a unique optimal investment plan $I^{\star}$ for problem \eqref{OC}.
	\end{theorem}
	
The previous theorem extends Theorem 2.3 in \cite{RS} in our general setting under Knightian uncertainty. Since its proof is developed through some auxiliary lemmata, we first discuss its main idea. The uniqueness of the optimal investment plan can be derived from the strict concavity of $\pi(x,\cdot)$, the affine structure of the production capacity with respect to the investment, and the strict comparison theorem for $g$-expectation.

In order to prove existence, a usual argument is to choose a suitable converging subsequence, which should exist by compactness of the set of admissible controls. Then, by the continuity of the payoff functional, the limit of such a subsequence is indeed optimal (whenever admissible). In our case, the main difficulty that we face is that it seems hard to find a topology ensuring at the same time compactness of the admissible set $\mathcal{A}_g$ and continuity of the payoff functional $\Pi$. To overcome this problem, as in \cite{RS} we first prove that it suffices to consider only those investment plans whose induced production capacity stays bounded from above by the overall maximal capacity under perfect reversibility (cf.\ Lemma \ref{ii3} below). The latter is integrable by Assumption \ref{a3}. Then, by choosing a maximizing sequence $\{I^n\}_{n\in\mathbb{N}}$ and applying the Koml\'{o}s' theorem (as in \cite{K99}), we conclude that there exists a subsequence $\{I_{n_k}\}_{k\in\mathbb{N}}$ whose arithmetic average suitably converges to some $I^{\star}$. Finally, it turns out that such $I^{\star}$ is indeed the optimal policy.
	
\begin{lemma}
	\label{ii2}
	Recall \eqref{OC}. Under Assumptions \ref{a1}, \ref{a2}, and \ref{a3}, the value $V^{\star}$ is finite.
	\end{lemma}
	
	\begin{proof}
		Simple manipulations on \eqref{i2} imply that
		\begin{equation}
		\label{i8}
		\begin{split}
		\int_0^T e^{-rt} \d I_t=&\int_0^T e^{-rt} \d C_t^I+\int_0^T \delta e^{-rt}C_t^I \d t\\
		=&e^{-rT}C_T^I - c + \int_0^T (\delta+r) e^{-rt} C_t^I \d t.
		\end{split}
		\end{equation}
	It thus follows that
		\begin{equation}
		\label{i0}
		\begin{split}
		& \int_0^T e^{-rt}\pi (X_t,C_t^I) \d t- \int_0^T e^{-rt} \d I_t \leq \int_0^T e^{-rt}\big(\pi (X_t,C_t^I) -(\delta+r)C_t^I\big) \d t + c\\
		& \leq \int_0^T e^{-rt} \pi^{\star}(X_t,r,\delta) \d t + c.
		\end{split}
		\end{equation}
		
		Now, standard estimates for BSDEs and the comparison theorem (see \cite{EPQ} and Proposition \ref{a21} in Appendix \ref{Appendix})) give that for any $I\in\mathcal{A}_g$ and for some $M>0$
		\begin{align*}
		\Pi(I) & \leq \mathcal{E}^{g}\bigg[\int_0^T e^{-rt} \pi^{\star}(X_t,r,\delta) \d t+ c\bigg] \leq \mathcal{E}^{\kappa}\bigg[\int_0^T \pi^{\star}(X_t,r,\delta) \d t+ c\bigg]\\
		&\leq \int_0^T \mathcal{E}^\kappa\big[\pi^{\star}(X_t,r,\delta)\big]\d t + c \leq M\Big(1 + \int_0^T \E\big[(\pi^{\star}(X_t,r,\delta))^2\big]^{\frac{1}{2}} \d t \Big) <\infty,
		\end{align*}
		where (H1) of Assumption \ref{a3} has been used. Since the last term on the right-hand side of the latter equation is independent of $I\in\mathcal{A}_g$, the desired result follows.
	\end{proof}

	
We now claim that we can restrict our analysis to those investment plans whose associated capacity is required to be below some overall maximal capacity under perfect reversibility. More precisely, recall $c^{\star}$ solving \eqref{cstar} and set
	\begin{equation}
	\label{Chat}
	\widehat{C}_t:=e^{-\delta t}\big[c + \sup_{s\in[0,t]}(c_s^{\star} e^{\delta s})\big],
	\end{equation}
	where, for simplicity, we have put $c_s^{\star}:=c^{\star}(X_s,r,\delta)$. Note that for any $t\in[0,T]$,
	$$
	\widehat{C}_t= ce^{-\delta t} + \sup_{s\in[0,t]}(c_s^{\star} e^{\delta(s-t)})\leq c + \sup_{s\in[0,t]}c_s^{\star} \leq c + \sup_{s\in[0,T]}c_s^{\star}.
	$$
	Then the investment plan corresponding to $\widehat{C}$, i.e.
	$$
	\widehat{I}_t=\widehat{C}_t - c + \int_0^t \delta \widehat{C}_s \d s, \quad t \geq 0,
  $$
	is admissible due to the fact that, for some $M>0$ (changing from line to line),
	\begin{equation}
	\label{i1}
	\mathcal{E}^{\widetilde{g}}\bigg[\int_0^T e^{-rt} \d\widehat{I}_t\bigg]\leq \mathcal{E}^\kappa\big[\widehat{I}_T\big]\leq M \mathcal{E}^\kappa\big[\sup_{s\in[0,T]}|c_s^{\star}|\big]\leq M \E\big[\sup_{s\in[0,T]}|c^{\star}_s|^2\big]^{\frac{1}{2}}<\infty.
	\end{equation}
	Here, the comparison theorem (cf.\ Proposition \ref{a21} in Appendix \ref{Appendix}) and standard estimates for BSDEs (cf.\ \cite{EPQ}), as well as (H2) of Assumption \ref{a3}, have been employed.

	\begin{lemma}
	\label{ii3}
		Set $\widehat{\mathcal{A}}_g:=\{I:\, I \in \mathcal{A}_g\,\,\text{such that}\,\,C^I_t \leq \widehat{C}_t\,\,\P_0\text{-a.s.}\,\,\forall t\geq 0\}$. Then we have
		\begin{displaymath}
		\sup_{I\in\mathcal{A}_g}\mathcal{E}^{g}\bigg[\int_0^T e^{-rt}\pi (X_t,C_t^I) \d t - \int_0^T e^{-rt} \d I_t\bigg]=
		\sup_{I\in\widehat{\mathcal{A}}_g}\mathcal{E}^{g}\bigg[\int_0^T e^{-rt}\pi (X_t,C_t^I) \d t - \int_0^T e^{-rt} \d I_t\bigg].
		\end{displaymath}
	\end{lemma}
	
	\begin{proof}
		It suffices to prove that for any $I\in\mathcal{A}_g$, we can find some $\bar{I}\in\widehat{\mathcal{A}}_g$ such that $\Pi(\bar{I})\geq \Pi(I)$. Let $I\in\mathcal{A}_g$ be given and fixed with associated capacity denoted, for simplicity, by $C$. Set $\bar{C}_t:=\min\{C_t,\widehat{C}_t\}$ and $\bar{A}_t:=e^{\delta t}\bar{C}_t$. Then, $\bar{C}$ is a production capacity associated to the investment plan $\bar{I}_t:=\int_0^t e^{-\delta s}\d\bar{A}_s$. Clearly, $\bar{I}\in\widehat{\mathcal{A}}_g$. We claim that $\Pi(\bar{I})\geq \Pi(I)$. By \eqref{i8} and the representation of $g$-expectation \eqref{repr-g} of Proposition \ref{prop:A1}, we obtain that
		\begin{align*}
		&\Pi(\bar{I})- \Pi(I)\\
		=&\mathcal{E}^{g}\bigg[\int_0^T e^{-rt}\big(\pi (X_t,\bar{C}_t)-(\delta+r)\bar{C}_t\big) \d t - e^{-rT}\bar{C}_T + c \bigg]\\
		&-\mathcal{E}^{g}\bigg[\int_0^T e^{-rt}\big(\pi (X_t,C_t)-(\delta+r)C_t\big) \d t -e^{-rT}C_T + c \bigg]\\
			=&\inf_{\theta\in D} \Big(\E^\theta\bigg[\int_0^T e^{-rt}\big(\pi (X_t,\bar{C}_t)-(\delta+r)\bar{C}_t\big)\d t-e^{-rT}\bar{C}_T\bigg]+\alpha_{0,T}(\theta)\Big)\\
		&-\inf_{\theta\in D} \Big(\E^\theta\bigg[\int_0^T e^{-rt}\big(\pi (X_t,C_t)-(\delta+r)C_t\big) \d t-e^{-rT}C_T\bigg]+\alpha_{0,T}(\theta)\Big)\\
		\geq &\inf_{\theta\in D} \E^\theta\bigg[\int_0^T e^{-rt}\big(\widetilde{\pi}(X_t,\bar{C}_t)-\widetilde{\pi}(X_t,C_t)\big) \d t-e^{-rT}\big(\bar{C}_T-C_T\big)\bigg],
		\end{align*}
		where $\widetilde{\pi}(x,c):=\pi(x,c)-(\delta+r)c$ is concave with respect to $c$. Now, on the set of times for which $C_t>\bar{C}_t=\widehat{C}_t \geq c_t^{\star}$, we have
		$$\widetilde{\pi}(X_t,\bar{C}_t)\geq \widetilde{\pi}(X_t,C_t),$$
		upon recalling that $c^{\star}$ attains the maximum of the concave function $\widetilde{\pi}(x,\cdot)$ (cf.\ \eqref{cstar}).
		On the other hand, on the set of times for which $C_t=\bar{C}_t$, it follows that $\widetilde{\pi}(X_t,\bar{C}_t)=\widetilde{\pi}(X_t,C_t)$. This, together with the fact that $\bar{C}_T\leq C_T$ $\P$-a.s.\ (hence, $\P^{\theta}$-a.s.\ for any $\theta \in D$) finally implies that $\Pi(\bar{I})\geq \Pi(I)$.
	\end{proof}
	
	\begin{remark}
	Notice that the result of Lemma \ref{ii3} still holds removing the requirement $g(t,\omega,0)=0$, for any $(t,\omega)\in[0,T]\times \Omega$.
	\end{remark}
	
	We are now finally able to prove Theorem \ref{ii1}.
	
	\begin{proof}[Proof of Theorem \ref{ii1}]
		We first show uniqueness. Assume that $I^1$ and $I^2$ are two distinguishable admissible optimal investment plans. Obviously, also $I:=\frac{1}{2}(I^1+I^2)$ is an admissible investment. By \eqref{i2}, we have $C^I_t=\frac{1}{2}(C_t^{I^1}+C_t^{I^2})$. Since $\pi(x,\cdot)$ is strictly concave, $\mathcal{E}^{g}[\cdot]$ is concave and satisfies the strict comparison theorem (cf.\ Proposition \ref{a21} in Appendix \ref{Appendix}), we can write
		\begin{align*}
		&\mathcal{E}^{g}\bigg[\int_0^T e^{-rt}\pi(X_t,C_t^I) \d t- \int_0^T e^{-rt} \d I_t\bigg]\\
		=&\mathcal{E}^{g}\bigg[\int_0^T e^{-rt}\pi(X_t,\frac{1}{2}(C_t^{I^1}+C_t^{I^2}))\d t- \int_0^T\frac{1}{2}e^{-rt}\big(\d I^1_t+ \d I_t^2\big)\bigg]\\
		>&\mathcal{E}^{g}\bigg[\frac{1}{2}\Big(\int_0^T e^{-rt}\pi(X_t,C_t^{I^1})\d t - \int_0^T e^{-rt} \d I^1_t +\int_0^T e^{-rt}\pi(X_t,C_t^{I^2})\d t - \int_0^T e^{-rt} \d I^2_t\Big)\bigg]\\
		\geq &\frac{1}{2}\Big(\Pi(I^1) + \Pi(I^2)\Big),
		\end{align*}
		which contradicts the optimality of $I^1$ and $I^2$. Hence, the optimal investment plan is unique (whenever it exists).
		\vspace{0.15cm}
		
		We are now in the position to show the existence. By Lemma \ref{ii3}, it is sufficient to consider only those $I$ such that $I\in\widehat{\mathcal{A}}_g$ since
		$$V^{\star}=\sup_{I\in\widehat{\mathcal{A}}_g} \Pi(I).$$
		Let $\{I^n\}_{n\in \mathbb{N}}\subset \widehat{\mathcal{A}}_g$ be a maximizing sequence; that is, such that $V^{\star}=\lim_{n\rightarrow\infty}\Pi(I^n)$. By Equation \eqref{i1} and noting that $I^n_T\leq \widehat{I}_T$ for any $n\in\mathbb{N}$, we have
		$$
		\sup_{n}\E^{\theta}[I^n_T]\leq \E^{\P^\theta}[\widehat{I}_T]\leq \mathcal{E}^\kappa[\widehat{I}_T]<\infty.
		$$
		for any $\P^{\theta}$ such that $\theta$ is progressively measurable and $|\theta|\leq \kappa$..
		Applying the Koml\'{o}s theorem (in the version of \cite{K99}; see \cite{K67} for the classical version), we can choose a subsequence, still denoted by $\{I^n\}_{n\in \mathbb{N}}$, such that the induced (optional) Borel-measure on $[0,T]$ converges in the Ces\'aro sense weakly to some measure $\d I^{\star}$ $\P^\theta$-a.s.; that is, by Portmanteau Theorem,
		$$\lim_{n\rightarrow\infty} J_t^n:=\lim_{n\rightarrow\infty}\frac{1}{n}\sum_{k=1}^n I_t^k = I_t^{\star},$$
		for all continuity points of $t \mapsto I_t^{\star}$ and for $t=T$ $\P^\theta$-a.s.
		Since $\P^\theta$ and $\P_0$ are equivalent, the above equation also holds $\P_0$-a.s. It thus follows that
		$$C^{I^{\star}}_t=\lim_{n\rightarrow\infty} C^{J^n}_t = \lim_{n\rightarrow\infty}\frac{1}{n}\sum_{k=1}^n C^{I^k}_t,$$
	for all continuity points of $t \mapsto C^{I^{\star}}_t$ and for $t=T$ $\P_0$-a.s. Combined with the fact that $C^{I^k}_t \leq \widehat{C}_t$ $\P_0$-a.s.\ for any $t\geq0$ and any $k\in \mathbb{N}$, the latter implies that $I^{\star}\in\widehat{\mathcal{A}}_g$.
	
	Since $\pi(x,\cdot)$ and $\mathcal{E}^{g}[\cdot]$ are both concave, we have by Jensen's inequality
		$$\Pi(J^n)\geq \frac{1}{n}\sum_{k=1}^n \Pi(I^k).$$
		Recalling now \eqref{i0} and noting that $\E[\int_0^T e^{-rt}|\pi^*(X_t,r,\delta)|^2 \d t]<\infty$, we can invoke Fatou's lemma (as in Proposition \ref{a21} in Appendix \ref{Appendix}) and obtain
	$$\Pi(I^{\star})\geq \limsup_{n\rightarrow\infty}\Pi(J^n)=V^{\star},$$
		which yields that $I^{\star}$ is an optimal investment plan.
	\end{proof}
	
\begin{remark}
	\label{r1}
		For any $\xi\in L^2(\mathcal{F}_T)$, consider general variational preferences given by
		$$\mathcal{E}[\xi]:=\inf_{\P\in\mathcal{P}}\big(\E^{\P}[\xi]+c(\P)\big).$$
		Here, $\mathcal{P}$ consists of all the probability measures $\P$ equivalent to $\P_0$ and such that
		$$\sup_{P\in\mathcal{P}}\E\Big[\Big|\frac{\d\P}{\d\P_0}\Big|^2\Big]<\infty,$$
		 and $c:\mathcal{P}\rightarrow \mathbb{R}$ is a penalty function such that $-\infty<\inf_{\P\in\mathcal{P}}c(\P)\leq \sup_{\P\in\mathcal{P}}c(\P)<\infty$.
		Also, let $\widetilde{\mathcal{E}}[\xi]:=-\mathcal{E}[-\xi]=\sup_{\P\in\mathcal{P}}\big(\E^{\P}[\xi]-c(\P)\big)$.
		
		We can then consider the irreversible investment problem where the company has to pick $I^{\star}\in \widetilde{\mathcal{A}}$ such that the following supremum is attained:
		\begin{equation}
		\label{e3}
		\sup_{I\in\widetilde{\mathcal{A}}}\mathcal{E}\bigg[\int_0^T e^{-rt}\pi(X_t,C_t^I) \d t - \int_0^T e^{-rt} \d I_t\bigg],
		\end{equation}
		where
		$$\widetilde{\mathcal{A}}:=\Big\{I\,\big|\,\widetilde{\mathcal{E}}\bigg[\int_0^T e^{-rt} \d I_t\bigg]<\infty,\,\, I \textrm{ is a nondecreasing, right-continuous and $\mathbb{F}$-adapted process}\Big\}.$$
		
		Under Assumptions \ref{a1} and \ref{a3}, all the results previously obtained in this section still hold, possibly with the exception of the uniqueness claim for the solution to \eqref{e3}. Indeed, $\mathcal{E}$ may not satisfy the strict comparison property. For example, in order to show that the value of \eqref{e3} is finite (cf.\ Lemma \ref{ii2}) one exploits the definition of $\mathcal{E}$ and \eqref{pistar} and obtains that for any $I\in\widetilde{\mathcal{A}}$
		\begin{align*}
			&\mathcal{E}\bigg[\int_0^T e^{-rt}\pi(X_t,C_t^I)\d t- \int_0^T e^{-rt}\d I_t\bigg] \leq \sup_{\P\in\mathcal{P}} \E^{\P}\bigg[\int_0^T e^{-rt}\pi^{\star}(X_t,r,\delta)\d t\bigg]-\inf_{\P\in\mathcal{P}}c(\P)\\
			& \leq \sup_{\P\in\mathcal{P}} \E\Big[\Big|\frac{\d\P}{\d\P_0}\Big|^2\Big]^{1/2}\int_0^T \E\big[\big|\pi^{\star}(X_t,r,\delta)\big|^2\big]^{1/2}\d t-\inf_{\P\in\mathcal{P}}c(\P) < \infty.
		\end{align*}
		\end{remark}

\begin{remark}
\label{Tinfty}
We could prove an existence and uniqueness result also for the case $T=+\infty$, by imposing suitable integrability requirements on $c^{\star}$ and $\pi^{\star}$ (cf.\ Assumption B.2 and Theorem B.3 in \cite{RS}). We refrain from doing that as a similar result will not be strictly needed in the subsequent analysis. Also, when the setting is time-homogeneous, one might be able to explicitly construct an optimal investment plan, as we show in the case study of Section \ref{sec:casestudy} below.
\end{remark}

\section{First-order Conditions and Base Capacity Policies}
\label{sec:FOCs}

In this section, we establish necessary and sufficient conditions for the optimality of an investment plan, and we then construct the optimal investment plan with the help of a suitable stochastic backward equation (see also \cite{BE} and \cite{RS}).

Recall Proposition \ref{prop:A1} and for any $I\in \mathcal{A}_g$ we set
	\begin{align*}
	\mathcal{P}_g(I) & :=\Big\{\P^\theta\,\big|\,\theta\in D, \mathcal{E}^{g}\bigg[\int_0^T e^{-rt}\pi (X_t,C^I_t)\d t - \int_0^T e^{-rt}\d I_t\bigg]\\
	& = \E^{\P^\theta}\bigg[\int_0^T e^{-rt} \pi (X_t,C^I_t) \d t - \int_0^T e^{-rt} \d I_t\bigg] + \alpha_{0,T}(\theta)\Big\}.
	\end{align*}
As a matter of fact, $\mathcal{P}_g(I)$ can be interpreted as the collection of all worst-case scenarios for $I$. The following first-order conditions for optimality then holds.

	\begin{theorem}
	\label{ii13}
		Under Assumptions \ref{a1} and \ref{a3}, and (i)-(iii) of Assumption \ref{a2}, an investment plan $I^{\star}$ is optimal for problem \eqref{OC} if it is admissible and there exists some $\P^\theta\in\mathcal{P}_g(I^{\star})$ such that the following conditions hold $\P^\theta$-a.s.:
			\begin{equation}
		\label{ii14}
		\begin{split}
		&\E^{\P^\theta}_t\bigg[\int_t^T e^{-(r+\delta)(s-t)}\pi_c(X_s,C^{\star}_s) \d s - 1\bigg] \leq 0, \quad \text{for all}\,\,t\geq 0, \\
		&\int_0^T e^{-rt} \E^{\P^\theta}_t\bigg[\int_t^T e^{-(r+\delta)(s-t)}\pi_c(X_s,C^{\star}_s) \d s - 1\bigg] \d I^{\star}_t = 0.
		\end{split}\end{equation}
		Here, $C^{\star}$ is the production capacity induced by $I^{\star}$ through \eqref{i2}.
		  \end{theorem}

\begin{proof}
		 Let $I \in \mathcal{A}_g$, recall \eqref{solC}, and, in order to simplify exposition, set $C=C^I$. Using Proposition \ref{prop:A1} we have for $P^\theta\in\mathcal{P}_g(I^{\star})$, $\theta \in D$,
		\begin{align}
		\label{verify}
		&\Pi(I^{\star})-\Pi(I) \nonumber \\
		=&\mathcal{E}^{g}\bigg[\int_0^T e^{-rt}\pi (X_t,C^{\star}_t)\d t - \int_0^T e^{-rt}\d I_t^{\star}\bigg]-\mathcal{E}^{g}\bigg[\int_0^T e^{-rt}\pi (X_t,C_t)\d t - \int_0^T e^{-rt}\d I_t\bigg] \nonumber \\
		\geq &\E^{\P^\theta}\bigg[\int_0^T e^{-rt}\pi (X_t,C^{\star}_t)\d t - \int_0^T e^{-rt}\d I_t^{\star}\bigg] + \alpha_{0,T}(\theta) \nonumber \\
		&\hspace{0.2cm} - \E^{\P^\theta}\bigg[\int_0^T e^{-rt}\pi (X_t,C_t)\d t - \int_0^T e^{-rt}\d I_t\bigg] -\alpha_{0,T}(\theta)  \\
		= &\E^{\P^\theta}\bigg[\int_0^T e^{-rt}\pi_c(X_t,C_t^{\star})\big(C_t^{\star}-C_t\big)\d t - \int_0^T e^{-rt}\d\big(I_t^{\star}-I_t\big)\bigg] \nonumber \\
		=&\E^{\P^\theta}\bigg[\int_0^T e^{-rt}\pi_c(X_t,C_t^{\star})\Big(\int_0^t e^{\delta(s-t)}\d\big(I_s^{\star}-I_s\big)\Big) \d t - \int_0^T e^{-rt}\d\big(I_t^{\star}-I_t\big)\bigg] \nonumber \\
		=&\E^{\P^\theta}\bigg[\int_0^T e^{-rt}\Big(\int_t^T e^{-(r+\delta)(s-t)}\pi_c(X_s,C_s^{\star})\d s - 1\Big)\d\big(I_t^{\star}-I_t\big)\bigg],\nonumber
		\end{align}
		where Fubini-Tonelli's theorem has been employed.
		
		By Theorem 1.33 in \cite{J79} and the second equation in \eqref{ii14} we obtain that
		\begin{align*}
		&\E^{\P^\theta}\bigg[\int_0^T e^{-rt}\Big(\int_t^T e^{-(r+\delta)(s-t)}\pi_c(X_s,C_s^{\star})\d s - 1\Big)\d I_t^{\star}\bigg] \\
		=&\E^{\P^\theta}\bigg[\int_0^T e^{-rt}\E^{\P^\theta}_t\bigg[\int_t^T e^{-(r+\delta)(s-t)}\pi_c(X_s,C_s^{\star})\d s-1\bigg]\d I_t^{\star}\bigg]=0.
		\end{align*}
		Also, by the first equation in \eqref{ii14} one has
		\begin{align*}
		&\E^{\P^\theta}\bigg[\int_0^T e^{-rt}\Big(\int_t^T e^{-(r+\delta)(s-t)}\pi_c(X_s,C_s^{\star})\d s - 1\Big)\d I_t\bigg] \\
		=&\E^{\P^\theta}\bigg[\int_0^T e^{-rt}\E^{\P^\theta}_t\bigg[\int_t^T e^{-(r+\delta)(s-t)}\pi_c(X_s,C_s^{\star})\d s-1\bigg]\d I_t\bigg]\leq 0.
		\end{align*}
		Feeding those last two equations back into \eqref{verify} we find $\Pi(I^{\star})\geq \Pi(I)$; that is, $I^{\star}$ is optimal.
	\end{proof}
	
	 \begin{remark}\label{r2}
		Consider the irreversible investment problem \eqref{e3} in Remark \ref{r1}. Let $I$ be an admissible investment plan and $C^I$ its associated production capacity. Set
		\begin{align}
		\mathcal{P}(I) &:=\Big\{\P\in\mathcal{P}\,\Big|\,\inf_{\P\in\mathcal{P}}\Big\{\E^{\P}\bigg[\int_0^T e^{-rt}\pi(X_t,C^I_t) \d t - \int_0^T e^{-rt}\d I_t\bigg]+c(\P)\Big\} \nonumber \\
		& = \E^{\P}\bigg[\int_0^T e^{-rt}\pi(X_t,C^I_t) \d t - \int_0^T e^{-rt}\d I_t\bigg]+c(\P)\Big\}. \nonumber
		\end{align}
		
		Then, an investment plan $I^{\star}$ is optimal if it is admissible and there exists some $\P\in\mathcal{P}(I^{\star})$ such that $\P$-a.s.:
		\begin{displaymath}
		\begin{split}
		&\E^{\P}_t\bigg[\int_t^T e^{-(r+\delta)(s-t)}\pi_c(X_s,C^{\star}_s) \d s - 1\bigg] \leq 0, \quad \text{for all}\,\,t\geq 0,\\
		&\int_0^T e^{-rt} \E^{\P}_t\bigg[\int_t^T e^{-(r+\delta)(s-t)}\pi_c(X_s,C^{\star}_s) \d s - 1\bigg] \d I^{\star}_t = 0.
		\end{split}
		\end{displaymath}
	\end{remark}
	
	We now prove that the first-order conditions \eqref{ii14} are also necessary for optimality under an additional assumption on the driver $g$. For the detailed proof, we may refer to Appendix \ref{AppendixC}.
		
		\begin{theorem}
		\label{ii15}
			Suppose that Assumptions \ref{a1}, \ref{a3} and Assumption \ref{a2} (i)-(iii) hold. Furthermore, assume that $g$ satisfies the following condition:
 \begin{itemize}
 \item[(v)] For each $\omega\in\Omega$, $t\in[0,T]$ and $z\in\mathbb{R}$, the equation $g(t,\omega,z)-xz=f(t,\omega,x)$ admits a unique solution $x\in[-\kappa,\kappa]$, denoted by $x(t,\omega,z)$. Furthermore, $z \mapsto x(t,\omega,z)$ is continuous, for any $(t,\omega) \in [0,T]\times\Omega$.
 \end{itemize}
  Suppose that for any constant $a>0$ and for any $t\in[0,T]$, one has $\E\big[|\pi_c(X_t,a)|^2\big]<\infty$.
 If an investment plan $I^{\star}$ is optimal for problem \eqref{OC}, then there exists some $\P^\theta\in\mathcal{P}_g(I^{\star})$ such that the following conditions hold $\P^\theta$-a.s.:
		\begin{displaymath}
		\begin{split}
		&\E^{\P^\theta}_t\bigg[\int_t^T e^{-(r+\delta)(s-t)}\pi_c(X_s,C^{\star}_s) \d s - 1\bigg] \leq 0, \quad \text{for all}\,\,t\geq 0, \\
		&\int_0^T e^{-rt} \E^{\P^\theta}_t\bigg[\int_t^T e^{-(r+\delta)(s-t)}\pi_c(X_s,C^{\star}_s) \d s - 1\bigg] \d I^{\star}_t = 0.
		\end{split}\end{displaymath}
		Here, $C^{\star}$ is the production capacity induced by $I^{\star}$ through \eqref{i2}.
	 \end{theorem}

\begin{remark}
\begin{enumerate}
\item The proof of the necessity of the first-order condition needs a careful and technical extension of the results in \cite{BR} in order to accommodate our multiple-priors setting. In particular, it requires estimates for BSDEs and condition (v) above (see Lemma \ref{l6}). That is the main reason why we cannot extend the necessary conditions to the general variational preferences, as instead we do for the sufficient conditions in Remark \ref{r2}.
\item The requirement $\E[|\pi_c(X_t,a)|^2]<\infty$, for any $a>0$ and $t\in[0,T]$, holds for the separable operating profit function discussed in Remark \ref{remark3.2}.
\end{enumerate}
\end{remark}

Motivated by the analysis of \cite{RS} and the sufficiency of the first-order conditions for optimality, we can now establish an appropriate construction for the optimal investment plan. A key building block is the so-called \emph{base capacity} which induces the minimal production capacity that the firm would keep. In other words, if the capacity is greater than the minimal one, no more investment would be made. If capacity is below the minimal one, the firm should invest ``just enough" in order to reach the minimal level.

\begin{definition}[cf.\ Definition 3.1 in \cite{RS}]
	\label{RS}
		For a given optional process $\ell$ and depreciation rate $\delta\geq 0$,
		\begin{equation}\label{ii4}
		C_t^{\ell,\delta}:=e^{-\delta t}\big[c \vee \sup_{s\in[0,t]}(\ell_s e^{\delta s})\big]
		\end{equation}
		is the capacity that tracks $\ell$ at depreciation rate $\delta$. Setting
		\begin{equation}
		\label{zeta}
		\zeta_t^{\ell,\delta}:= \sup_{s\in[0,t]}\Big(\frac{\ell_s - c e^{-\delta s}}{e^{-\delta s}}\Big) \vee 0, \quad \zeta_{0^-}^{\ell,\delta}:=0,
		\end{equation}
		the investment plan that finances $C^{\ell,\delta}$ is denoted by $I^{\ell,\delta}$ and it is such that
		\begin{equation}
		\label{ii5}
		I_t^{\ell,\delta}:= \int_0^t e^{-\delta s} \d \zeta_s^{\ell,\delta}, \quad  I_{0^-}^{\ell,\delta}:=0.
		\end{equation}
		We call $I^{\ell,\delta}$ the base capacity policy with depreciation rate $\delta$ and base capacity $\{\ell_t\}_{t\in[0,T]}$.
		\end{definition}


Due to ambiguity, the irreversible investment problem is considered under multiple priors $\{\P^\theta,\theta\in D\}$ (cf.\ Proposition \ref{prop:A1}). Therefore, the base capacity changes if one chooses a different belief $\P^{\theta}$. We shall see below that, in order to obtain an explicit description for the optimal policy of \eqref{OC}, we are faced with the problem of finding a probability measure $\P^{\theta}$ and an $\mathbb{F}$-progressively measurable process $\ell^{\P^{\theta}}:=\ell^{\theta}$ such that: (i) $\ell^{\theta}$ is a solution to a suitable backward equation under $\P^{\theta}$ and, at the same time, (ii) under $\P^{\theta}$ the net expected profit resulting from an implementation of the base capacity policy $I^{\ell^{\theta},\delta}$ is minimal.

	\begin{theorem}
	\label{i6}
	For each $\P^\theta$ with $\theta\in D$ there exists a unique $\mathbb{F}$-progressively measurable process $\{\ell^\theta_t\}_{t\in [0,T]}$ that solves the backward equation
		\begin{equation}
		\label{ii6}
		\E^{\P^\theta}_t\bigg[\int_t^T e^{-(r+\delta)(s-t)}\pi_c\big(X_s,e^{-\delta s}\sup_{t\leq u\leq s}\big(\ell^\theta_u e^{\delta u}\big)\big) \d s\bigg]=1
		\end{equation}
		for all $t<T$. Let $C^{\ell^\theta,\delta}$ be the capacity that tracks $\ell^\theta$ at depreciation rate $\delta$ and $I^{\ell^\theta,\delta}$ be the base capacity policy with depreciation rate $\delta$ and base capacity $\ell^\theta$. Then $I^{\ell^\theta,\delta}$ is optimal for the irreversible investment problem \eqref{OC} if it is admissible and
		\begin{align}
		\label{ii12}
		& \mathcal{E}^{g}\bigg[\int_0^T e^{-rt}\pi (X_t,C_t^{\ell^\theta,\delta})\d t - \int_0^T e^{-rt} \d I_t^{\ell^\theta,\delta})\bigg] \nonumber \\
		& = \E^{\P^\theta}\bigg[\int_0^T e^{-rt}\pi (X_t,C_t^{l^\theta,\delta})\d t - \int_0^T e^{-rt} \d I_t^{\ell^\theta,\delta}\bigg] + \alpha_{0,T}(\theta).
		\end{align}
	\end{theorem}
	
	\begin{proof}
		For the existence and uniqueness of the solution to backward equation \eqref{ii6}, we may refer to the proof of Theorem 3.2 in \cite{RS}. In fact, Bank and El Karoui \cite{BE} study this kind of backward equation in a more general setting.
		
		To complete the proof, it suffices to show that $C^{\ell^\theta,\delta}$ satisfies the sufficient first-order conditions \eqref{ii14}. Note that $\pi$ is strictly concave in its second variable and
		\begin{displaymath}
		C^{\ell^\theta,\delta}_s \geq e^{-\delta s}\sup_{t\leq u\leq s}\big(\ell^\theta_u e^{\delta u}\big), \quad \forall s \in [0,T].
		\end{displaymath}
		It thus follows that
		\begin{align*}
		&\E^{\P^\theta}_t\bigg[\int_t^T e^{-(r+\delta)(s-t)}\pi_c(X_s,C^{\ell^\theta,\delta}_s) \d s\bigg]
		\leq \E^{\P^\theta}_t\bigg[\int_t^T e^{-(r+\delta)(s-t)}\pi_c(X_s,e^{-\delta s}\sup_{t\leq u\leq s}(\ell^\theta_u e^{\delta u}))\d s\bigg]=1.
		\end{align*}
		That is, the first equation of \eqref{ii14} is satisfied.
		
		In particular, if $t$ belongs to the support of the random Borel-measure $\d I^{\ell^\theta,\delta}_t$,
		for any $s\geq t$, we have
		\begin{displaymath}
		C^{\ell^\theta,\delta}_s= e^{-\delta s}\big[c \vee \sup_{0\leq u\leq s}\big(\ell^\theta_u e^{\delta u}\big)\big]= e^{-\delta s}\sup_{t\leq u\leq s}\big(\ell^\theta_u e^{\delta u}\big),
		\end{displaymath}
		and therefore also the second of \eqref{ii14} holds.
	\end{proof}

\begin{remark}
\label{r3}
Consider the general irreversible investment problem \eqref{e3} in Remark \ref{r1}. Let $\ell^{\P}$ be the solution to \eqref{ii6} under probability $\P\in\mathcal{P}$ and let $C^{\ell^{\P},\delta}$, $I^{\ell^{\P},\delta}$ be, respectively, the capacity and investment policy associated to base capacity $\ell^{\P}$ and depreciation rate $\delta$. Then $I^{\ell^{\P},\delta}$ is optimal for \eqref{e3} if $I^{\ell^{\P},\delta}\in\widetilde{\mathcal{A}}$ and $\P\in \mathcal{P}(I^{\ell^{\P},\delta})$, where $\mathcal{P}(I^{\ell^{\P},\delta})$ is given in Remark \ref{r2}.
\end{remark}



In the following, we provide some interesting properties of the optimal investment plan. First of all, Proposition \ref{p1} below can be viewed as a ``martingale optimality principle'': that is to say, an investment plan which is optimal at the original time is its best continuation at any time afterwards.

Before proceeding, we need some preliminary definition and material.
\begin{definition}
\label{def:StronSM}
Let $\mathcal{T}_{0,T}$ be the collection of all $\mathbb{F}$-stopping times taking values between $0$ and $T$. A family of random variables $\{X_\tau,\tau\in\mathcal{T}_{0,T}\}$ is said to be an {\rm{$\mathcal{E}^g$-supermartingale in the strong sense}}, if $X_\tau \in L^2(\mathcal{F}_\tau)$ and $\mathcal{E}^g_\tau[X_\sigma] \leq X_\tau$ for any $\tau,\sigma\in\mathcal{T}_{0,T}$ with $\tau\leq \sigma$.
\end{definition}

Given the optimal investment plan $I^{\star}$ for \eqref{OC}, let $S\in \mathcal{T}_{0,T}$. We set
\begin{align}
\label{AS}
\mathcal{A}_S(I^{\star}):=&\Big\{I\,\Big|\, I \textrm{ is nondecreasing, right-continuous and $\mathbb{F}$-adapted, } \nonumber \\
&\ \  I|_{[0,S)}=I^{\star}|_{[0,S)},\,\,{\E}^{\widetilde{g}}\Big[\int_0^T e^{-rt} \d I_t\Big]<\infty\Big\}.
\end{align}

\begin{proposition}
\label{p1}
Suppose that Assumptions \ref{a1}, \ref{a2} and \ref{a3} hold. Furthermore, assume that $g$ is super-additive. Let $I^{\star}$ be optimal for \eqref{OC}, recall \eqref{AS}, and consider then the optimal irreversible investment problem
\begin{equation}
\label{ii18}
V_S:=\esssup_{I\in\mathcal{A}_S(I^{\star})}\mathcal{E}^{g}_S\bigg[\int_0^T e^{-rt}\pi(X_t,C^{I}_t) \d t - \int_0^T e^{-rt}\d I_t\bigg], \quad S\in\mathcal{T}_{0,T}.
\end{equation}
Then, $\big\{V_S,\, S\in\mathcal{T}_{0,T}\big\}$ is an $\mathcal{E}^{g}$-supermartingale in the strong sense. Moreover, $I^{\star}$ is also optimal for \eqref{ii18}.
\end{proposition}

\begin{proof}
We start by proving the supermartingale property. We prove it only for deterministic times, since the proof for stopping times is similar. For simplicity, Let
\begin{displaymath}
\mathcal{J}(I):=\int_0^T e^{-rt}\pi(X_t,C^{I}_t) \d t- \int_0^T e^{-rt} \d I_t,
\end{displaymath}
Now for any $0\leq s\leq t \leq T$, we claim that there exists a sequence $\{I^n\}_{n\in\mathbb{N}}\subset\mathcal{A}_t(I^{\star})\subset\mathcal{A}_s(I^{\star})$ such that $\mathcal{E}^{g}_t[\mathcal{J}(I^n)]$ is increasing in $n$ and
	 \begin{displaymath}
	 V_t=\lim_{n\rightarrow\infty}\mathcal{E}^{g}_t\big[\mathcal{J}(I^n)\big].
 \end{displaymath}
Such a claim is actually a consequence of Lemma \ref{l1} below. By Proposition \ref{A2} it follows
	 \begin{align*}
	 \mathcal{E}^{g}_s\big[V_t\big]&=\mathcal{E}^{g}_s\big[\lim_{n\rightarrow\infty} \mathcal{E}^{g}_t\big[\mathcal{J}(I^n)\big]\big]=\lim_{n\rightarrow\infty}\mathcal{E}^{g}_s\big[\mathcal{E}^{g}_t\big[\mathcal{J}(I^n)\big]\big]\\
	 &=\lim_{n\rightarrow\infty}\mathcal{E}^{g}_s\big[\mathcal{J}(I^n)\big]\leq \esssup_{I\in\mathcal{A}_s(I^{\star})}\mathcal{E}^{g}_s\big[\mathcal{J}(I^n)\big]=V_s.
	 \end{align*}
	 Therefore, $\{V_t\}_{t\in[0,T]}$ is an $\mathcal{E}^{g}$-supermartingale.
	
	We now move one by showing the second claim of the proposition; that is, the optimal investment plan for problem \eqref{OC}, $I^{\star}$, is also optimal for \eqref{ii18}. Arguing as in the proof of Theorem \ref{ii1}, we can show that there exists a unique investment plan $I^{S,\star}$ which is optimal for problem \eqref{ii18}.

Suppose now that $I^{S,\star}$ and $I^{\star}$ are distinguishable on $[S,T]$. Then we have
	 \begin{displaymath}
	 	\mathcal{E}_S^{g}\big[\mathcal{J}(I^{S,\star})\big]>\mathcal{E}_S^{g}\big[\mathcal{J}(I^{\star})\big],
	 \end{displaymath}
	 which, by the strict comparison theorem for $g$-expectation (cf.\ Proposition \ref{a21} in Appendix \ref{Appendix}), gives
	 \begin{displaymath}
	 \mathcal{E}^{g}\big[\mathcal{J}(I^{S,\star})\big]=\mathcal{E}^{g}\big[\mathcal{E}_S^{g}\big[\mathcal{J}(I^{S,\star})\big]\big]>\mathcal{E}^{g}\big[\mathcal{E}_S^{g}\big[\mathcal{J}(I^{\star})\big]\big]=\mathcal{E}^{g}\big[\mathcal{J}(I^{\star})\big].
	 \end{displaymath}
	 Hence a contradiction is reached since $I^{S,\star}\in\mathcal{A}_g$ and the proof is complete.
\end{proof}

\begin{lemma}
\label{l1}
Under the same assumption as Proposition \ref{p1}. Let
$I^{\star}$ be optimal for \eqref{OC}, and recall \eqref{AS}. Then for any $0\leq s\leq t\leq T$, the family $\big\{\mathcal{E}^{g}_t\big[\mathcal{J}(I)\big],\, I \in \mathcal{A}_t(I^{\star})\big\}$ is upwards directed\footnote{That is, for any $(I^1,I^2) \in \mathcal{A}_t(I^{\star})$ there exists $I \in \mathcal{A}_t(I^{\star})$ such that $\mathcal{E}^{g}_t\big[\mathcal{J}(I)\big] \geq \max\big\{\mathcal{E}^{g}_t\big[\mathcal{J}(I^1)\big],\mathcal{E}^{g}_t\big[\mathcal{J}(I^2)\big]\big\}$.} and $\mathcal{A}_t(I^{\star})\subseteq \mathcal{A}_s(I^{\star})$.
\end{lemma}

\begin{proof}
We first show that the family $\big\{\mathcal{E}^{g}_t\big[\mathcal{J}(I)],\, I\in\mathcal{A}_t(I^{\star})\big\}$ is upwards directed. For any $I^1,I^2\in\mathcal{A}_t(I^{\star})$, set $A:=\big\{\omega \in \Omega:\, \mathcal{E}^{g}_t\big[\mathcal{J}(I^1)\big](\omega)\geq \mathcal{E}^{g}_t\big[\mathcal{J}(I^2)\big](\omega)\big\}$ and $I:= I^1 \mathds{1}_{A}+I^2 \mathds{1}_{A^c}$. It is easy to check that $I\in\mathcal{A}_t(I^{\star})$ by the super-additivity of $g$.

By \eqref{i2}, we obtain that $C^I=C^{I^1}\mathds{1}_A+C^{I^2}\mathds{1}_{A^c}$. Then, Assumption \ref{a1}-(3) implies that $\pi(X_t,C^I_t)=\pi(X_t,C^{I^1}_t)\mathds{1}_A+\pi(X_t,C^{I^2}_t)\mathds{1}_{A^c}$. Hence, we have $\mathcal{J}(I)=\mathcal{J}(I^1)\mathds{1}_A + \mathcal{J}(I^2)\mathds{1}_{A^c}$. This fact, together with  Proposition \ref{A2} in Appendix \ref{Appendix}, yield that
\begin{align*}
\mathcal{E}^{g}_t\big[\mathcal{J}(I)\big]& = \mathcal{E}^{g}_t\big[\mathcal{J}(I^1)\mathds{1}_A + \mathcal{J}(I^2)\mathds{1}_{A^c}\big] = \mathcal{E}^{g}_t\big[\mathcal{J}(I^1)\big]\mathds{1}_A + \mathcal{E}^{g}_t\big[\mathcal{J}(I^2)\big]\mathds{1}_{A^c} \\
&=\max\big\{\mathcal{E}^{g}_t\big[\mathcal{J}(I^1)\big],\mathcal{E}^{g}_t\big[\mathcal{J}(I^2)\big]\big\}.
\end{align*}
Hence the claim follows and, in particular, we may choose an increasing sequence $\big\{\mathcal{E}^{g}_t\big[\mathcal{J}(I^n)\big]\big\}_{n\in\mathbb{N}}$ such that
\begin{displaymath}
\esssup_{I\in\mathcal{A}_t(I^{\star})}\mathcal{E}^{g}_t\big[\mathcal{J}(I)\big] = \lim_{n\rightarrow\infty}\mathcal{E}^{g}_t\big[\mathcal{J}(I^n)\big].
\end{displaymath}
Clearly, for any $s\leq t$, $\mathcal{A}_t(I^{\star})\subseteq \mathcal{A}_s(I^{\star})$. The proof is complete.
\end{proof}

Assuming that the economic shock $X$ is deterministic, we may conjecture that the optimal policy is deterministic as well. It means that, when the firm has full knowledge about the future economic conditions, also the investment plan is certain. This is in fact proven in the next proposition.
\begin{proposition}
Let $h:E \times \R \to \R$ be defined as
\begin{displaymath}
		h(x,\ell):=\begin{cases}
		\pi_c(x,-\frac{e^{-\delta t}}{\ell})e^{-(r+\delta)t}, & \ell<0,\\
		-\ell, &\ell\geq 0.
		\end{cases}
	\end{displaymath}
and assume that $X$ is deterministic and such $\int_0^T |h(X_t,\ell)| \d t < \infty$ for any $\ell \in\mathbb{R}$. Suppose also that the driver $g$ satisfies Assumption \ref{a2}.

Then, the optimal investment plan is also deterministic.
\end{proposition}

\begin{proof}
	By Theorem 2 in \cite{BE}, there exists a unique function $\ell':[0,T)\rightarrow\mathbb{R}\cup\{-\infty\}$ such that
	\begin{displaymath}
		\int_t^T h(X_s,\ell'_s) \d s = e^{-(r+\delta)t}.
	\end{displaymath}
	For any $\theta\in D$, set $\ell^{\theta}_t=-\frac{e^{-(r+\delta)t}}{\ell'_t}$. Let $C^{\ell^\theta,\delta}$ be the capacity that tracks $\ell^\theta$ at depreciation rate $\delta$, and $I^{\ell^\theta,\delta}$ be the base capacity policy with depreciation rate $\delta$ and base capacity $\ell^\theta$. Clearly, being $\ell'$ deterministic, also all the above processes are deterministic. Moreover,
	\begin{displaymath}
	\E^{P^\theta}_t\bigg[\int_t^T e^{-(r+\delta)s}\pi_c(X_s,e^{-\delta s}\sup_{t\leq u\leq s}(\ell_u^\theta e^{\delta u}))\d u\bigg]=e^{-(r+\delta)t}.
	\end{displaymath}
	However, it also holds that
	\begin{align*}
	& \mathcal{E}^{g}\bigg[\int_0^T e^{-rt} \pi (X_t,C_t^{\ell^\theta,\delta})\d t - \int_0^T e^{-rt}\d I_t^{\ell^\theta,\delta}\bigg] \\
&= \int_0^T e^{-rt} \pi (X_t,C_t^{\ell^\theta,\delta})\d t - \int_0^T e^{-rt}\d I_t^{\ell^\theta,\delta}\\
	& = \E^{P^\theta}\bigg[\int_0^T e^{-rt} \pi (X_t,C_t^{\ell^\theta,\delta})\d t - \int_0^T e^{-rt}\d I_t^{\ell^\theta,\delta}\bigg],
	\end{align*}
where we have used the constant preserving property for the $g$-expectation in the first equality. Hence, the desired result then follows from Theorem \ref{i6}.
\end{proof}


\section{Explicit Solution in an Homogeneous Setting with Infinite Horizon}
\label{sec:casestudy}

\subsection{Setting}
\label{sec:settingcase}

In this section we provide the explicit solution to the irreversible investment problem \eqref{OC} in the following setting:
\begin{assumption}
\label{ass:case}
\begin{itemize}
\hspace{10cm}
\item[(i)] $T=+\infty$ and $\delta=0$;
\vspace{0.15cm}

\item[(ii)] under $\P_0$,
$$X^x_t = x \exp\Big(\big(b - \frac{1}{2}\sigma^2\big) t + \sigma B_t\Big), \quad t\geq 0,\,\,x>0,$$
for some $b\in\R$ and $\sigma>0$;
\vspace{0.15cm}

\item[(iii)] $g(t,z)=-\kappa|z|,$ for any $(t,z)\in \R_+ \times \R$ and for some $\kappa>0$.
\vspace{0.15cm}

\item[(iv)] $\pi(x,c) = \frac{1}{1-\alpha} x^{\alpha}c^{1-\alpha}$, for some elasticity $\alpha\in(0,1)$.
\vspace{0.15cm}

\item[(v)]  $r> b + \sigma\kappa + \frac{1}{2}(\sigma + \kappa)^2$.
\end{itemize}
\end{assumption}

\begin{remark}
Notice that condition (v) above is consistent with the condition proposed in \cite{RS} (see Example 7.3 therein), that can indeed be obtained by taking $\kappa=0$.
\end{remark}

Our approach will be to exploit the sufficiency of the first-order conditions for the optimality (that actually also holds for $T=+\infty$; cf.\ Theorem \ref{ii13}) and construct a candidate optimal solution with the help of a suitable stochastic backward equation in the spirit of Theorem \ref{i6}.

Let
$$\Xi^\kappa=\Big\{\{\xi_t\}_{t\geq 0}\,\Big|\, \xi \textrm{ is $\mathbb{F}$-progressively measurable and } |\xi_t| \leq \kappa,\,\,\forall t\geq 0\,\,\P_0\textrm{-a.s.}\Big\},$$
and for any $\xi\in\Xi^\kappa$ and $t\geq 0$ set
\begin{displaymath}
\epsilon^{\xi}_t:= \exp\Big(\int_0^t \xi_s \d B_s-\frac{1}{2}\int_0^t\xi^2_s \d s\Big).
\end{displaymath}
The latter process defines a probability measure $\P^{\xi}$ such that $\frac{\d\P^{\xi}}{\d\P_0}\Big|_{\mathcal{F}_t}=\epsilon^{\xi}_t$. In the following, we will denote by $\E^{\xi}$ the expectation under $\P^{\xi}$, for any given $\xi\in\Xi^\kappa$.
Finally, let $\mathcal{A}_\kappa^\infty$ be the collection of all nondecreasing, right-continuous and $\mathbb{F}$-adapted processes $I$ such that  $I_{0^-}= 0$ and
\begin{displaymath}
\sup_{\xi\in\Xi^\kappa}\E^{\xi}\bigg[\int_0^\infty e^{-rt} \d I_t\bigg]<\infty.
\end{displaymath}

Within such a framework our aim is to solve
\begin{equation}
\label{e1}
\sup_{I\in\mathcal{A}^\infty_\kappa}\inf_{\xi\in\Xi^\kappa}\E^{\xi}\bigg[\int_0^\infty e^{-rt} \pi(X^x_t,C^{c,I}_t) \d t - \int_0^\infty e^{-rt} \d I_t\bigg],
\end{equation}
where, as usual, $C^{c,I}$ is the production capacity induced by any admissible investment plan $I$ with initial capacity $c$ (cf.\ \eqref{solC}, but now with $\delta=0$; i.e.\ $C^{c,I}_t=c + I_t$).

For any $I \in \mathcal{A}^\infty_\kappa$ set
\begin{align*}
& \Xi(I):=\Big\{\xi\in\Xi^{\kappa}\,\Big|\,\inf_{\xi'\in\Xi^\kappa}\E^{\xi'}\bigg[\int_0^\infty e^{-rt} \pi(X^x_t,C^{c,I}_t) \d t - \int_0^\infty e^{-rt} \d I_t\bigg] \\
& = \E^{\xi}\bigg[\int_0^\infty e^{-rt} \pi(X^x_t,C^{c,I}_t) \d t - \int_0^\infty e^{-rt} \d I_t\bigg]\Big\}.
\end{align*}
In fact, $\Xi(I)$ can be regarded as the collection of worst-case kernels for the given investment plan $I$.

Similarly to Theorem \ref{ii13} and Theorem \ref{i6} (whose proofs indeed do hold also for $T=\infty$), we have the following results.
	\begin{theorem}
	\label{}
	An investment plan $I^{\star}$ is optimal if there exists some $\xi\in\Xi(I^{\star})$ such that the following conditions hold $\P^{\xi}$-a.s.
	\begin{equation}
	\label{thm:casestudy1}
	\begin{split}
	&\E^{\xi}_t\bigg[\int_t^\infty e^{-r(s-t)}\pi_c(X^x_s,C^{c,I^{\star}}_s) \d s\bigg]\leq 1, \quad \textrm{ for any } t \geq 0,\\
	&\int_0^{\infty}\Big(\E^{\xi}_t\bigg[\int_t^\infty e^{-r(s-t)}\pi_c(X^x_s,C^{c,I^{\star}}_s) \d s\bigg] - 1 \Big) \d I^{\star}_t = 0.
	\end{split}
	\end{equation}
	\end{theorem}

\begin{theorem}
\label{thm:casestudy2}
	For each $\xi\in \Xi^\kappa$, let $\{\ell^\xi_t\}_{t\in[0,T]}$ be the unique $\mathbb{F}$-progressively measurable process such that it solves the following backward equation
	\begin{equation}
	\label{e2}
	\E^{\xi}_t\bigg[\int_t^\infty e^{-r(s-t)}\pi_c\big(X^x_s,\sup_{t\leq u\leq s}\big(\ell^\xi_u\big)\big) \d s\bigg] = 1,
	\end{equation}
	for all $t\geq 0$. Let $C^{c,\ell^\xi}:=C^{c,\ell^\xi,0}$ be the capacity that tracks $\ell^\xi$ at depreciation rate $\delta=0$ with initial value $c$ and $I^{\ell^\xi}:=I^{\ell^\xi,0}$ be the base capacity policy with depreciation rate $\delta=0$ and base capacity $\ell^\xi$. Then $I^{\ell^\xi}$ is optimal for the irreversible investment problem \eqref{e1} if it is admissible and $\xi\in \Xi(I^{\ell^\xi})$; that is,
	\begin{equation}
	\label{}
	\inf_{\xi'\in\Xi^\kappa}\E^{\xi'}\bigg[\int_0^\infty e^{-rt}\pi(X^x_t,C_t^{c,\ell^\xi})\d t - \int_0^\infty e^{-rt} \d I_t^{\ell^\xi}\bigg]=\E^{\xi}\bigg[\int_0^\infty e^{-rt}\pi(X^x_t,C_t^{c,\ell^\xi})\d t - \int_0^\infty e^{-rt} \d I_t^{\ell^\xi}\bigg].
	\end{equation}
\end{theorem}


\subsection{Solving Problem \eqref{e1}}
\label{sec:solution}

For any $x>0$, $c\geq 0$, let
\begin{equation}
\label{eq:OCkappa}
v(x,c):=\sup_{I\in\mathcal{A}_{-\kappa}}\E^{-\kappa}\bigg[\int_0^\infty e^{-rt} \pi(X^x_t,C^{c,I}_t) \d t - \int_0^\infty e^{-rt} \d I_t\bigg],
\end{equation}
where $\mathcal{A}_{-\kappa}$ is the set of all nondecreasing, $\mathbb{F}$-adapted, right-continuous processes such that $I_{0^-}=0$ and $\E^{-\kappa}[\int_0^\infty e^{-rt} \d I_t]<\infty$. Denote by $I^{\star,-\kappa}$ the corresponding optimal policy and by $C^{\star,-\kappa,c}$ the induced optimal capacity.

We shall prove that
\begin{align}
\label{e1-bis}
& \sup_{I\in\mathcal{A}^\infty_\kappa}\inf_{\xi\in\Xi^\kappa}\E^{\xi}\bigg[\int_0^\infty e^{-rt} \pi(X^x_t,C^{c,I}_t) \d t - \int_0^\infty e^{-rt} \d I_t\bigg] \nonumber
\\
& = \E^{-\kappa}\bigg[\int_0^\infty e^{-rt} \pi(X^x_t,C^{\star,-\kappa,c}_t) \d t - \int_0^\infty e^{-rt} \d I^{\star,-\kappa}_t\bigg].
\end{align}
Our plan is now the following:

\begin{enumerate}
\item Solve problem \eqref{eq:OCkappa} and then show that $I^{\star,-\kappa}\in\mathcal{A}^\infty_\kappa$. This is accomplished in Section \ref{sec:solvingOCkappa} below.

\item Show, via a probabilistic verification argument, that \eqref{e1-bis} indeed holds. This is done in Section \ref{sec:verifying} below.
\end{enumerate}


\subsubsection{Solving \eqref{eq:OCkappa}}
\label{sec:solvingOCkappa}

Notice that under $\P^{-\kappa}$, the process $X^x$ of Assumption \ref{ass:case}-(ii) is still a geometric Brownian motion, but with drift $b-\sigma \kappa$; that is, under $\P^{-\kappa}$ one has for any $t\geq0$
$$X^x_t = x \exp\Big(\big(b - \sigma\kappa - \frac{1}{2}\sigma^2\big) t + \sigma B^{-\kappa}_t\Big), \quad x>0,$$
where, for any $T>0$, $B^{-\kappa}_t:=B_t + \kappa t$, $t\in[0,T]$, is a Brownian motion under $\P^{-\kappa}$  according to Girsanov theorem.

In order to determine the optimal irreversible investment plan for problem \eqref{eq:OCkappa} we can rely on the result of Proposition 7.1 in \cite{RS}, where \eqref{eq:OCkappa} is solved under the more general case of an exponential L\'evy dynamics for the process $X$. In particular, the following holds.
\begin{theorem}
\label{thm:t1}
 	Under Assumption \ref{ass:case}, $\ell^{-\kappa}$ is the optimal base capacity for problem \eqref{eq:OCkappa}, where $\ell^{-\kappa}_t=KX^x_t$, $t\geq 0$, and
 \begin{equation}
\label{K-k}
 	K=K_{-\kappa}:=\left(\E^{-\kappa}\bigg[\int_0^\infty e^{-rs} \inf_{0\leq u\leq s} \Big(\frac{{X^{1}_{s}} }{{X^{1}_{u}}}\Big)^\alpha \d s\bigg]\right)^{\frac{1}{\alpha}}.
 \end{equation}
 Hence,
\begin{equation}
\label{I-k}
I^{\star,-\kappa}_t:= \sup_{0\leq s\leq t}\big(K X^x_s - c\big) \vee 0, \quad t \geq 0, \qquad I^{\star,-\kappa}_{0^-}=0,
\end{equation}
is optimal for problem \eqref{eq:OCkappa},and the induced optimal production capacity is
\begin{equation}
\label{C-k}
C^{\star,-\kappa,c}_t:= c \vee \sup_{0\leq s\leq t}\big(K X_s\big), \quad t \geq 0, \qquad C^{\star,-\kappa,c}_{0^-}=c.
\end{equation}
\end{theorem}

Notice that the irreversible investment plan $I^{\star,-\kappa}$ is such that
\begin{equation}
\label{Sk1}
C^{\star,-\kappa,c}_t \geq K X^x_t, \quad  \forall t \geq 0 \quad \P^{-\kappa}-a.s.
\end{equation}
Moreover, it is a standard result that
\begin{equation}
\label{Sk2}
I^{\star,-\kappa}_t = \int_0^t \mathds{1}_{\{C^{\star,-\kappa,c}_s \leq K X^x_s\}} \d I^{\star,-\kappa}_s, \quad t \geq 0 \quad \P^{-\kappa}-a.s.;
\end{equation}
that is, if $(C^{\star,-\kappa,c}_t, X^x_t)$ belong to $\{(x',c')\in\R\times\R_+:\, c' \leq K x'\}$, then such a time $t$ is a time of increase for $I^{\star,-\kappa}$.

It is also easy to see from \eqref{K-k} that (cf.\ \cite{handbook})
\begin{equation}
\label{K-k-bis}
K^\alpha = \frac{1}{r}\E^{-\kappa}\bigg[\int_0^{\infty} re^{-rt} e^{\alpha \underline{Z}_t} \d t\bigg] = \frac{1}{r}\E^{-\kappa}\Big[e^{\alpha \underline{Z}_{\tau_r}}\Big] = \frac{1}{r} \frac{\beta_{-}}{\beta_{-} - \alpha\sigma^2},
\end{equation}
where $\tau_r$ is an independent random time, exponentially distributed with parameter $r$,
$$\beta_{-}:=- (b - \sigma \kappa-\frac{1}{2}{\sigma^2}) - \sqrt{\big(b - \sigma \kappa-\frac{1}{2}\sigma^2\big)^2 + 2r\sigma^2},$$
and $\underline{Z}_t = \inf_{0\leq s \leq t}\big((b - \sigma \kappa - \frac{1}{2}\sigma^2) t + \sigma B^{-\kappa}_t\big)$.

\begin{remark}
\label{rem:prop}
\begin{enumerate}
\item By Lemma \ref{ii3} we know that the optimal investment plan $I^{\star,-\kappa}$ should be such that the corresponding capacity $C^{\star,-\kappa,c}$ is not larger than $\widehat{C}$, where $\widehat{C}_t= c + \sup_{0 \leq s \leq t}c^{\star}_s$, with $c^{\star}_s$ satisfying
\begin{displaymath}
\pi_c(X^x_s,c^{\star}_s)=r, \quad s\geq0.
\end{displaymath}
It is easy to see that in our setting $c^{\star}_s=\frac{X^x_s}{r^{1/\alpha}}$. Since $K^\alpha \leq \frac{1}{r}$, it readily follows that $C^{\star,-\kappa,c} \leq \widehat{C}$ as desired.

\item It is easy to check that $K$ from \eqref{K-k-bis} is decreasing with respect to the interest rate $r$. Therefore, such is the optimal investment plan under $\P^{-\kappa}$ as well.
\end{enumerate}
\end{remark}

The following lemma shows that the investment plan $I^{\star,-\kappa}$ induced by the base capacity $\ell^{-\kappa}$ is admissible for the initial problem \eqref{e1}; i.e., it belongs to $\mathcal{A}_{\kappa}^{\infty}$.

\begin{lemma}
\label{admissible}
Under Assumption \ref{ass:case}, we have
\begin{displaymath}
\sup_{\xi\in \Xi^\kappa}\E^\xi\left[\int_0^\infty  e^{-rt}\d I^{\star,-\kappa}_t\right]<\infty.
\end{displaymath}
\end{lemma}

\begin{proof}

Since for any $\xi\in\Xi^\kappa$
\begin{displaymath}
\E^\xi\left[\int_0^\infty  e^{-rt}\d I^{\star,-\kappa}_t\right]=\E\left[\int_0^\infty e^{-rt}\epsilon^\xi_t\d I^{\star,-\kappa}_t\right],
\end{displaymath}
an integration by parts gives
	\begin{align}
	\label{admiss0}
	\int_0^T e^{-rt}\epsilon^\xi_t\d I_t^{\star,-\kappa}&=\int_0^T e^{-rt}  \epsilon^\xi_t \d C_t^{\star,-\kappa,c} \nonumber \\
	&= e^{-rT} \epsilon_T^\xi C_T^{\star,-\kappa,c}-c-\int_0^T C^{\star,-\kappa,c}_t \d (e^{-rt}\epsilon^\xi_t) \\
	&= e^{-rT} \epsilon_T^\xi C_T^{\star,-\kappa,c}-c +\int_0^T r e^{-rt} \epsilon^\xi_t C_t^{\star,-\kappa,c}\d t+\int_0^T  e^{-rt} \xi_t \epsilon^\xi_t C_t^{\star,-\kappa,c}\d B_t, \nonumber
	\end{align}
	for any $T>0$ given and fixed.
	Let $\tilde{C}^{-\kappa}_t:=\sup_{s\in[0,t]}\ell^{-\kappa}_s =\sup_{s\in[0,t]}(KX^x_s)$. It is easy to check that $C^{\star,-\kappa,c}_t\leq c+ \tilde{C}^{-\kappa}_t$. Hence,
	\begin{align}
	\label{admiss1}
	& \E\left[\int_0^\infty r e^{-rt} \epsilon^\xi_t C_t^{\star,-\kappa,c}\d t\right] \leq c + \E\left[\int_0^\infty r e^{-rt} \epsilon^\xi_t \tilde{C}_t^{-\kappa}\d t\right]= c + K x \E\left[\int_0^\infty  r e^{-rt+\bar{Z}^\xi_t}\d t\right],
	\end{align}
	where $\bar{Z}^{\xi}_t=\sup_{0\leq s\leq t}Z^\xi_s$ and $Z^\xi_t=\sigma B_t+(b-\frac{1}{2}\sigma^2)t+\sigma\int_0^t \xi_sds$. Since, for any $\xi\in \Xi^\kappa$,
 \begin{displaymath}
 \E\left[\int_0^\infty  r e^{-rt+\bar{Z}^\xi_t}\d t\right]\leq \E\left[\int_0^\infty  r e^{-rt+\bar{Z}^\kappa_t}\d t\right],
 \end{displaymath}
we can continue from \eqref{admiss1} and write
\begin{equation}
	\label{admiss2}
\E\left[\int_0^\infty r e^{-rt} \epsilon^\xi_t C_t^{\star,-\kappa,c}\d t\right] \leq c + K x \E\left[\int_0^\infty r e^{-rt+\bar{Z}^\kappa_t}\d t\right].
\end{equation}

Let now $\tau_r$ be an independent exponentially distributed random time with parameter $r$ and $\beta_+:=\beta_{+}(\kappa)$ be the largest solution to the equation $\frac{1}{2}\sigma^2 \beta^2+(b-\frac{1}{2}\sigma^2+\sigma\kappa)\beta-r=0$. Notice that, by Assumption \ref{ass:case}, $\beta_{+}>1$. Then, by Equation (1.1.1) in \cite{handbook}, we can write
	\begin{displaymath}
		\E\left[\int_0^\infty re^{-rt+\bar{Z}^\kappa_t}\d t\right]=\E\left[\exp(\bar{Z}^\kappa_{\tau_r})\right]= \frac{\beta_{+}}{\beta_{+}-1}.
	\end{displaymath}
The latter, together with \eqref{admiss2}, imply
 \begin{equation}
\label{admiss3}
 \E\left[\int_0^\infty r e^{-rt} \epsilon^\xi_t {C}_t^{\star,-\kappa,c}\d t\right] \leq c + K x \frac{\beta_{+}}{\beta_{+}-1}.
 \end{equation}

 By arguments similar to those employed in the proof of Lemma 4.9 in \cite{BR}, we also obtain that
	\begin{equation}
	\label{e4}
		\lim_{T\rightarrow\infty}e^{-rT} \epsilon^\xi_T C_T^{\star,-\kappa,c}=0, \quad \P_0-\text{a.s.},
	\end{equation}
	which together with \eqref{admiss0} and \eqref{admiss3} give
	\begin{align*}
	\E\left[\int_0^\infty \epsilon^\xi_t e^{-rt}\d I_t^{\star,-\kappa}\right]=\E\left[\int_0^\infty r \epsilon^\xi_t e^{-rt}C_t^{\star,-\kappa,c}\d t\right]-c
	\leq    K x \frac{\beta_{+}}{\beta_{+}-1}.
	\end{align*}
	Therefore, $I^{\star,-\kappa}$ belongs to $\mathcal{A}_{\kappa}^{\infty}$.
	\end{proof}

We now move on by determining the expression for $v(x,c)$. For this we slightly generalize the results of Proposition 7.2 in \cite{RS} to the case in which $c$ is not necessarily null as assumed therein. The detailed proof can be found in Appendix \ref{AppendixB}.
\begin{proposition}
\label{prop:v}
The value function of problem \eqref{eq:OCkappa} is such that $v\in C^{2,1}(\R_+ \times [0,\infty);\R)$ and it is given by
\begin{equation}
\label{eq:vcase}
v(x,c)=
\begin{cases}
c - K x + v(x, K x), & c \leq K x \\
\frac{x^\alpha c^{1-\alpha}}{(1-\alpha) \tilde{r} }
+ \frac{1}{\lambda-1}
 x^\lambda   c^{1-\lambda} K^\mu
 \left( \frac{1}{\tilde{r}} - K^\alpha\right)& c > K x,
\end{cases}
\end{equation}
where $\lambda$ is the largest solution to $\phi(\lambda)=r$, with $\phi(\lambda)=\frac{1}{2}\sigma^2\lambda^2+(b-\frac{1}{2}\sigma^2-\sigma \kappa)\lambda$, $\mu=\lambda-\alpha$, and $\tilde{r}=r-\phi(\alpha)$.
\end{proposition}

\begin{remark}
\label{rem:onv}
\begin{enumerate}
\item Since $I^{\star,-\kappa}$ is optimal for problem \eqref{eq:OCkappa}, one can write
\begin{displaymath}
v(x,c)=\E^{-\kappa}\bigg[\int_0^\infty e^{-rt} \pi(X^x_t,C^{\star,-\kappa,c}_t) \d t - \int_0^\infty e^{-rt} \d I^{\star,-\kappa}_t\bigg].
\end{displaymath}
Then, recalling equations \eqref{I-k} and \eqref{C-k}, it is easy to check that for any positive constant $a$ we have $v(ax,ac)=av(x,c)$, where the homogeneity of the Cobb-Douglas profit function has also been employed.

\item Because the economic shock $X^x$ is linear with respect to $x$ (being a geometric Brownian motion), and $x \mapsto \pi(x,c)$ is increasing, it is easy to see from \eqref{eq:OCkappa} that $v_x(x,c) \geq 0$ for any $(x,c) \in \R_+ \times [0,\infty)$. This observation will be of fundamental importance in the proof of Theorem \ref{thm:maincase} below.
\end{enumerate}
\end{remark}


\subsection{Verifying that $\P^{-\kappa}$ is the worst-case scenario}
\label{sec:verifying}

For any $\xi\in\Xi^\kappa$ given and fixed, define for any $t\geq0$,
\begin{align}
\label{Hxi}
& H^{\star,\xi}_t:=\E^{\xi}_t\bigg[\int_t^{\infty} e^{-rs} \pi(X^x_s, C^{\star,-\kappa,c}_s) \d s - \int_t^{\infty} e^{-rs} \d I^{\star,-\kappa}_s\bigg] \nonumber \\
& = \frac{1}{\epsilon^{\xi}_t}\E_t\bigg[\int_t^{\infty} e^{-rs} \epsilon^{\xi}_s \pi(X^x_s, C^{\star,-\kappa,c}_s) \d s - \int_t^{\infty} e^{-rs} \epsilon^{\xi}_s \d I^{\star,-\kappa}_s\bigg]
\end{align}
where $I^{\star,-\kappa}$ and $C^{\star,-\kappa,c}$ are as in \eqref{I-k} and \eqref{C-k}.  We shall prove in Theorem \ref{thm:maincase} below that, for any $\xi\in \Xi^\kappa$, one $H^{\star,\xi}_0\geq H^{\star,-\kappa}_0$. The next technical result will be needed in order to accomplish that.

\begin{lemma}\label{square-integrable}
Under Assumption \ref{ass:case}, the random variable
$$\int_0^{\infty} e^{-rs} \epsilon^{\xi}_s\pi(X^x_s, C^{\star,-\kappa,c}_s) \d s - \int_0^{\infty} e^{-rs} \epsilon^{\xi}_s \d I^{\star,-\kappa}_s$$ is square-integrable under $\P_0$, for any $\xi\in \Xi^\kappa$.
\end{lemma}

\begin{proof}
By \eqref{admiss0} and \eqref{e4}, we have
\begin{align*}
&\int_0^{\infty} e^{-rs}\epsilon_s^\xi \pi(X^x_s, C^{\star,-\kappa,c}_s) \d s - \int_0^{\infty} e^{-rs}\epsilon_s^\xi \d I^{\star,-\kappa}_s\\
=&\int_0^{\infty} e^{-rs}\epsilon_s^\xi \pi(X^x_s, C^{\star,-\kappa,c}_s) \d s-\int_0^\infty r e^{-rs}\epsilon_s^\xi C^{\star,-\kappa,c}_s\d s-\int_0^\infty e^{-rs}\xi_s  \epsilon_s^\xi C^{\star,-\kappa,c}_s\d B_s+ c\\
\leq &\int_0^\infty e^{-rs} \epsilon_s^\xi\pi^{\star}(X^x_s,r,0) \d s-\int_0^\infty e^{-rs} \xi_s \epsilon_s^\xi C^{\star,-\kappa,c}_s\d B_s+ c\\
=& \tilde{\alpha}r^{-\frac{1}{\tilde{\alpha}}} \int_0^\infty e^{-rs}\epsilon_s^\xi X_s^x\d s-\int_0^\infty e^{-rs} \xi_s \epsilon_s^\xi C^{\star,-\kappa,c}_s\d B_s+ c
\end{align*}
and
\begin{align*}
&\int_0^{\infty} e^{-rs}\epsilon_s^\xi \pi(X^x_s, C^{\star,-\kappa,c}_s) \d s - \int_0^{\infty} e^{-rs} \epsilon_s^\xi \d I^{\star,-\kappa}_s\\
\geq &-\int_0^\infty r e^{-rs}\epsilon_s^\xi C^{\star,-\kappa,c}_s\d s-\int_0^\infty e^{-rs} \xi_s  \epsilon_s^\xi C^{\star,-\kappa,c}_s\d B_s\\
\geq &-\int_0^\infty re^{-rs}\epsilon_s^\xi\big(c+\sup_{u\in[0,s]}(KX^x_u)\big)\d s-\int_0^\infty  e^{-rs} \xi_s \epsilon_s^\xi C^{\star,-\kappa,c}_s\d B_s,
\end{align*}
where $\tilde{\alpha}:=\frac{\alpha}{\alpha-1}$.

Let $\bar{X}_s^x:=\sup_{u\in[0,s]}X_u^x$. A simple calculation implies that
\begin{displaymath}
(X_s^x)^2=x^2 \exp\big((2b-\sigma^2)s+2\sigma B_s\big).
\end{displaymath}
Also, it is easy to check that, for some $\varepsilon>0$ and $M>0$, by H\"older's inequality
\begin{align*}
\E \left[\Big(\int_0^\infty e^{-rs}\epsilon_s^\xi \bar{X}^x_s\d s\Big)^2\right]&\leq \E\left[\int_0^\infty e^{-2\varepsilon rs}\d s \int_0^\infty e^{-2(1-\varepsilon)rs}(\epsilon_s^\xi)^2(\bar{X}^x_s)^2\d s\right]\\
&\leq M \int_0^\infty e^{-2(1-\varepsilon)rs}\E\big[(\epsilon_s^\xi)^2(\bar{X}^x_s)^2\big]\d s.
\end{align*}
Since $|\xi_s|\leq \kappa$ for any $s\geq 0$, we have for any $\xi\in \Xi^\kappa$
\begin{align*}
& \E\big[(\epsilon_s^\xi)^2(\bar{X}^x_s)^2\big]=\E^{2\xi}\big[\exp(\int_0^s \xi_u^2 \d u)(\bar{X}^x_s)^2\big]\leq e^{\kappa^2 s}\E^{2\xi}\big[(\bar{X}^x_s)^2\big] \\
& \leq e^{\kappa^2 s} \E^{2\kappa}[(\bar{X}^x_s)^2] \leq x^2 e^{\kappa^2 s} e^{(\sigma^2+2(2\sigma\kappa+b))s},
\end{align*}
where Girsanov theorem has also been used. By Assumption \ref{ass:case}-(v), we may choose $\varepsilon$ small enough such that $2(1-\varepsilon)r> \sigma^2+\kappa^2+2(2\sigma\kappa+b)$, which in turn gives
$$\E\bigg[\Big(\int_0^\infty e^{-rs}\epsilon^\xi_s\bar{X}^x_s\d s\Big)^2\bigg]<\infty.$$

By It\^{o}'s isometry and similar arguments as those employed above one can also show that
\begin{displaymath}
\E\bigg[\left(\int_0^\infty \xi_s e^{-rs}\epsilon_s^\xi C^{\star,-\kappa,c}_s\d B_s\right)^2\bigg]=\E\bigg[\int_0^\infty e^{-2rs} \xi_s^2 \epsilon_s^{2\xi} (C^{\star,-\kappa,c}_s)^2\d s\bigg]<\infty.
\end{displaymath}
The proof is then complete.
\end{proof}

We can now prove the main result of this section ensuring that $\P^{-\kappa}$ realizes the worst case scenario.

\begin{theorem}
\label{thm:maincase}
Recall \eqref{Hxi}. For any $\xi\in \Xi^\kappa$, we have $H^{\star,\xi}_0 \geq H^{\star,-\kappa}_0$. Hence, \eqref{e1-bis} holds and $\P^{-\kappa}$ is the worst case scenario.
\end{theorem}

\begin{proof}
The proof is organized in four steps.
\vspace{0.25cm}

\emph{Step 1.} We start by proving that $H^{\star,-\kappa}_t=e^{-rt}v(X_t^x,C^{\star,-\kappa,c}_t)$, $t\geq 0$, with $v$ as in \eqref{eq:vcase}.

By \eqref{i8} and \eqref{e4}, we have
\begin{equation}
\label{eq:parts}
\int_t^{\infty} e^{-rs} \d I^{\star,-\kappa}_s=\int_t^\infty re^{-rs}C^{\star,-\kappa,c}_s\d s-e^{-rt}C^{\star,-\kappa,c}_t.
\end{equation}
Consequently,
\begin{align*}
v(x,c)=&\E^{-\kappa}\bigg[\int_0^{\infty} e^{-rs} \pi(X^x_s, C^{\star,-\kappa,c}_s) \d s-\int_0^\infty re^{-rs}C^{\star,-\kappa,c}_s\d s+c\bigg]\\
=&\E^{-\kappa}\bigg[\int_0^{\infty} e^{-rs} \pi(X^x_s, c\vee\sup_{t\in[0,s]}(KX^x_t)) \d s-\int_0^\infty re^{-rs}\{c\vee\sup_{t\in[0,s]}(KX^x_t)\}\d s+c\bigg].
\end{align*}
Notice that for any $u\geq 0$ and $t\geq 0$, we have $C^{\star,-\kappa,c}_{u+t}=C_t^{\star,-\kappa,c}\vee \sup_{s\in[0,u]}(KX^x_{s+t})$. By the Markov property and the homogeneity of $\pi$, we then obtain that
\begin{align*}
&\E^{-\kappa}_t\bigg[\int_t^{\infty} e^{-rs} \pi(X^x_s, C^{\star,-\kappa,c}_s) \d s\bigg] = \E^{-\kappa}_t\bigg[\int_0^{\infty} e^{-r(s+t)} \pi(X^x_{s+t}, C^{\star,-\kappa,c}_{s+t}) \d s\bigg]\\
=&e^{-rt}X_t^x\E^{-\kappa}_t\bigg[\int_0^{\infty} e^{-rs} \pi\big(\frac{X^x_{s+t}}{X_t^x}, \frac{C^{\star,-\kappa,c}_t}{X_t^x}\vee \frac{\sup_{u\in[0,s]}(KX^x_{u+t})}{X_t^x}\big) \d s \bigg]\\
=&e^{-rt}X_t^x\E^{-\kappa}_t\bigg[\int_0^{\infty} e^{-rs} \pi\big(\frac{X^x_{s+t}}{X_t^x},y\vee \frac{\sup_{u\in[0,s]}(KX^x_{u+t})}{X_t^x}\big) \d s \bigg]\bigg|_{y=\frac{C^{\star,-\kappa,c}_t}{X_t^x}}\\
=&e^{-rt}X_t^x\E^{-\kappa}\bigg[\int_0^{\infty} e^{-rs} \pi(X_s^1,y\vee \sup_{u\in[0,s]}(KX^1_{u})) \d s \bigg]\bigg|_{y=\frac{C^{\star,-\kappa,c}_t}{X_t^x}}.
\end{align*}
By using \eqref{eq:parts} and by following arguments similar to the previous ones also for the term $\E^{\xi}_t[\int_t^{\infty} e^{-rs} \d I^{\star,-\kappa}_s]$, we finally find
\begin{align*}
H^{\star,-\kappa}_t=&e^{-rt}X_t^x\E^{-\kappa}\bigg[\int_0^{\infty} e^{-rs} \pi(X_s^1,y\vee \sup_{u\in[0,s]}(KX^1_{u})) \d s\\ &-\int_0^\infty re^{-rs}\big(y\vee \sup_{u\in[0,s]}(KX^1_{u})\big)\d s+y \bigg]\bigg|_{y=\frac{C^{\star,-\kappa,c}_t}{X_t^x}}\\
=&e^{-rt}X_t^x v(1, \frac{C^{\star,-\kappa,c}_t}{X_t^x})=e^{-rt}v(X_t^x, C^{\star,-\kappa,c}_t),
\end{align*}
where the last equality is due to Remark \ref{rem:onv}-(1).
\vspace{0.25cm}

\emph{Step 2.} Let now $\xi \in \Xi^\kappa$ be given and fixed. From \eqref{Hxi}, one can clearly write
\begin{equation}
\begin{split}
\label{Hxi-bis}
H^{\star,\xi}_t = \frac{1}{\epsilon^{\xi}_t}\Big(M^{\xi}_t - \int_0^t e^{-rs} \epsilon^{\xi}_s \pi(X^x_s, C^{\star,-\kappa,c}_s) \d s + \int_0^{t} e^{-rs} \epsilon^{\xi}_s \d I^{\star,-\kappa}_s\Big),
\end{split}
\end{equation}
where we have defined the square-integrable martingale (cf.\ Lemma \ref{square-integrable})
$$M^{\xi}_t:= \E_t\bigg[\int_0^{\infty} e^{-rs} \epsilon^{\xi}_s \pi(X^x_s, C^{\star,-\kappa,c}_s) \d s - \int_0^{\infty} e^{-rs} \epsilon^{\xi}_s \d I^{\star,-\kappa}_s\bigg].$$

By Proposition 4.18 in Chapter 3.4 of \cite{KS}, there exists an $\mathbb{F}$-progressively measurable process $Z^{\xi}$ such that $\E[\int_0^{\infty} |Z^{\xi}_t|^2 \d t] < \infty$ and $M^{\xi}_t = M^{\xi}_0 + \int_0^t Z^{\xi}_s \d B_s$. Then, set $N^\xi_t:= M^{\xi}_t - \int_0^t e^{-rs} \epsilon^{\xi}_s \pi(X^x_s, C^{\star,-\kappa,c}_s) \d s + \int_0^{t} e^{-rs} \epsilon^{\xi}_s \d I^{\star,-\kappa}_s$.  Recalling that $\epsilon^\xi_t=\exp(\int_0^t \xi_s\d B_s-\frac{1}{2}\int_0^t \xi_s^2\d s)$, we find that
$$\d\big(\frac{1}{\epsilon^\xi_t}\big)=\frac{1}{\epsilon^\xi_t}\big[-\xi_t \d B_t + \xi^2_t \d t\big],$$
and applying It\^o's product rule one obtains under $\P_0$
\begin{equation}
\label{eq:dynHxi}
\d H^{\star,\xi}_t = \widetilde{Z}^{\xi}_t \big(\d B_t - \xi_t \d t\big) - e^{-rt}\big(\pi(X^x_t, C^{\star,-\kappa,c}_t) \d t - \d I^{\star,-\kappa}_t\big).
\end{equation}
Here, $\widetilde{Z}^{\xi}_t:=(Z^{\xi}_t-\xi_t N^\xi_t)/\epsilon^{\xi}_t$.
\vspace{0.25cm}

\emph{Step 3.} We now identify $\widetilde{Z}^{-\kappa}$. Let $T>0$ be given and fixed and recall that $B^{-\kappa}_t:= B_t + \kappa t$, $t\in[0,T]$, is an $\mathbb{F}$-Brownian motion under $\P^{-\kappa})$, by Girsanov theorem. Then, under $\P^{-\kappa}$ we can write from \eqref{eq:dynHxi} that
\begin{equation}
\label{eq:dynHxi-k}
H^{\star,-\kappa}_T - H^{\star,-\kappa}_0 = \int_0^T \widetilde{Z}^{-\kappa}_t \d B^{-\kappa}_t - \int_0^T e^{-rt}\big(\pi(X^x_t, C^{\star,-\kappa,c}_t) \d t - \d I^{\star,-\kappa}_t\big).
\end{equation}

On the other hand, by \emph{Step 1} above, $H^{\star,-\kappa}_t=e^{-rt}v(X^x_t, C^{\star,-\kappa,c}_t)$; in particular, $H^{\star,-\kappa}_0=v(x,c)$. Hence, with reference to Proposition \ref{prop:v}, we can apply It\^o-Meyer's formula for semimartingales to the process $\{e^{-rt}v(X^x_t, C^{\star,-\kappa,c}_t)\}_{t\geq 0}$ on the time time-interval $[0,T]$, $T>0$, and also using \eqref{Sk1} and \eqref{Sk2} one finds that $\P^{-\kappa}$ a.s.
$$H^{\star,-\kappa}_T - H^{\star,-\kappa}_0 = \int_0^T \sigma e^{-rt} X^x_t v_x(X^x_t, C^{\star,-\kappa,c}_t) \d B^{-\kappa}_t - \int_0^T e^{-rt}\big(\pi(X^x_t, C^{\star,-\kappa,c}_t) \d t - \d I^{\star,-\kappa}_t\big).$$
Comparing the latter with \eqref{eq:dynHxi-k} we find that
\begin{equation}
\label{ident-Zk}
\widetilde{Z}^{-\kappa}_t = \sigma e^{-rt} X^x_t v_x(X^x_t, C^{\star,-\kappa,c}_t), \quad \P^{-\kappa}\otimes \d t\text{-a.e.\ on}\,\,(\Omega,\mathcal{F}_T)\times ([0,T],\mathcal{B}_{[0,T]}).
\end{equation}
\vspace{0.25cm}

\emph{Step 4.} For any $\xi\in \Xi^\kappa$, set $\widehat{H}^\xi_t=H^{\star,\xi}_t-H^{\star,-\kappa}_t$, $\widehat{Z}^\xi_t=\widetilde{Z}^{\xi}_t-\widetilde{Z}^{-\kappa}_t$. It is easy to check that under $\P_0$ one has
\begin{align*}
\d \widehat{H}^\xi_{t} = \widehat{Z}^\xi_t \d B_t - \xi_s \widehat{Z}^\xi_t \d t - \big(\kappa + \xi_t\big) \widetilde{Z}^{-\kappa}_t \d t.
\end{align*}

Let now $T>0$ be given and fixed. Moving back to the measure $\P^{\xi}$ (equivalent to $\P_0$ on $\mathcal{F}_T$),  we find from the latter
\begin{align*}
\widehat{H}^\xi_{T\wedge \tau^\xi_n} - \widehat{H}^\xi_0 = \int_0^{T\wedge \tau^\xi_n} \widehat{Z}^\xi_t \d B^{\xi}_t - \int_0^{T\wedge \tau^\xi_n} \big(\kappa + \xi_t\big) \widetilde{Z}^{-\kappa}_t \d t,
\end{align*}
where $B^{\xi}_t=B_t - \int_0^t \xi_s \d s$, $t\in[0,T]$, is a Brownian motion under $\P^{\xi}$, by Girsanov theorem, and
\begin{displaymath}
\tau^\xi_n:=\inf\{t\geq 0: \int_0^t (\widehat{Z}_s^\xi)^2 ds>n\}, \quad n \in \mathbb{N}.
\end{displaymath}

Taking now expectations under $\P^{\xi}$, changing measure and using \eqref{ident-Zk} we obtain
\begin{align*}
\widehat{H}^\xi_0 & = \E^{\xi}[\widehat{H}^\xi_{T\wedge \tau^\xi_n}] + \E^{\xi}\bigg[\int_0^{T\wedge \tau^\xi_n} \big(\kappa + \xi_t\big) \widetilde{Z}^{-\kappa}_t \d t\bigg] \\
& = \E^{\xi}[\widehat{H}^\xi_{T\wedge \tau^\xi_n}] + \E^{-\kappa}\bigg[\int_0^{T\wedge \tau^\xi_n} \epsilon^{\kappa+\xi}_t \big(\kappa + \xi_t\big) \widetilde{Z}^{-\kappa}_t \d t\bigg] \geq \E^{\xi}[\widehat{H}^\xi_{T\wedge \tau^\xi_n}].
\end{align*}
Thanks to Lemma \ref{square-integrable}, we now first let $n$ go to infinity; then, by taking limits as $T \uparrow \infty$ as well and by using \eqref{Hxi} we finally find $\widehat{H}^\xi_0 \geq 0$.
\end{proof}


\subsubsection{Properties of the Optimal Solution}
\label{sec:propcase}

By Theorem \ref{thm:t1}, the optimal base capacity is determined by the economic shock $X$, the interest rate $r$ and the parameter $\kappa$. In fact, $\kappa$ reflects the ambiguity faced by the agent. Roughly speaking, the larger $\kappa$, the more ambiguity. In this subsection, we study how the optimal base capacity depends on these parameters.

\begin{lemma}
\label{lem:l2}
For any constant $-\kappa \leq \xi \leq \kappa$, let
$$K_{\xi}:= \E^{\xi}\left[\int_0^\infty e^{-rs} \inf_{0\leq u\leq s} \big(\frac{{X^{1}_{s}}}{{X^{1}_{u}} }\big)^\alpha \d s\right].$$
Then, for any two constant $\xi^1$, $\xi^2$ such that $\xi^1 \leq \xi^2$, we have $K_{\xi^1}\leq K_{\xi^2}$.
\end{lemma}

\begin{proof}
	By the Tonelli's theorem, we have
	\begin{align*}
	K_{\xi}=\E^{\xi}\bigg[\int_0^\infty e^{-rs} \inf_{0\leq u\leq s} \big(\frac{{X^{1}_{s}}}{{X^{1}_{u}} }\big)^\alpha \d s\bigg]=\int_0^\infty e^{-rs}\E^{\xi}\Big[ \inf_{0\leq u\leq s} \big(\frac{{X^{1}_{s}} }{{X^{1}_{u}}}\big)^\alpha \Big]\d s =: \int_0^\infty e^{-rs} f(s,\xi)\d s.
	\end{align*}
By changing measure and applying Girsanov Theorem
	\begin{align*}
	f(s,\xi)=&\E^{\xi}\Big[ \inf_{0\leq u\leq s} \exp\big(\alpha\sigma(B_s-B_u)+\alpha(b-\frac{1}{2}\sigma^2)(s-u)\big) \Big]\\
	=&\E\Big[\inf_{0\leq u\leq s} \exp\big(\alpha\sigma(B_s-B_u)+\alpha(b-\frac{1}{2}\sigma^2)(s-u)+\alpha\sigma \xi(s-u)\big) \Big].
	\end{align*}
	It is easy to check that $f(s,\xi)$ is increasing in $\xi$. The proof is complete.
\end{proof}

\begin{remark}
\label{rem:compstat}
	Consider another irreversible investment problem in which now the economic shock is given by $Y$ evolving as
	\begin{displaymath}
		\d Y_t^{y}=\mu Y_t^{y} \d t+\sigma Y_t^{y}\d B_t, \ Y_0^{y}=y.
	\end{displaymath}
	The base capacity corresponding to the economic shocks $X$ and $Y$ are denoted by $\ell^X$ and $\ell^Y$, respectively. Now, if $x\leq y$ and $b\leq \mu$, we know that $X^{x}_t\leq Y^{y}_t$, for any $t\geq 0$. By Theorem \ref{thm:t1}, we have $\ell^X=K^X X_t^{x}$ and $\ell^Y=K^Y Y_t^{y}$, where
	\begin{align*}
		&K^X=\bigg(\E\left[\int_0^\infty e^{-rs}\epsilon^{-\kappa}_{s} \inf_{0\leq u\leq s} \big(\frac{{X^{1}_{s}} }{{X^{1}_{u}} }\big)^\alpha \d s\right]\bigg)^{\frac{1}{\alpha}},\\
		&K^Y=\bigg(\E\left[\int_0^\infty e^{-rs}\epsilon^{-\kappa}_{s} \inf_{0\leq u\leq s} \big(\frac{{Y^{1}_{s}} }{{Y^{1}_{u}} }\big)^\alpha ds\right]\bigg)^{\frac{1}{\alpha}}.
	\end{align*}
	Applying the arguments of the proof of Lemma \ref{lem:l2} yields that $K^X\leq K^Y$. Consequently, we have $\ell^X\leq \ell^Y$.
\end{remark}

\begin{remark}
	Consider now the irreversible investment problem
	\begin{displaymath}
		\sup_{I\in\mathcal{A}^\infty_{\kappa_i}}\inf_{\xi\in\Xi^{\kappa_i}}\E^\xi\bigg[\int_0^\infty  e^{-rt}\pi(X^{x}_t,C^{c,I}_t) \d t \int_0^\infty  e^{-rt} \d I_t)\bigg], \quad i=1,2,
	\end{displaymath}
	where $\kappa_1\leq \kappa_2$.
	 It is natural to conjecture that when the firm faces more ambiguity, the scale of investment would be decreased. Lemma \ref{lem:l2} indeed implies that the constant $K_\xi$ is increasing in $\xi$. Hence, the conjecture follows easily from Theorem \ref{thm:t1}.
\end{remark}


\section*{Acknowledgments}
\noindent Financial support by the German Research Foundation (DFG) through the Collaborative Research Centre 1283 ``Taming uncertainty and profiting from randomness and low regularity in analysis, stochastics and their applications'' is gratefully acknowledged.


\appendix

\section{}
\label{Appendix}

\renewcommand{\theequation}{A-\arabic{equation}}

In this section, we introduce some fundamental results related to $g$-expectation, which is derived from the solution to a BSDE. Let $(\Omega,\mathcal{F},\mathbb{F}:=\{\mathcal{F}_t\}_{t\in[0,T]},\P_0)$ be a complete filtered probability space, with filtration satisfying the usual conditions.

We consider
\begin{itemize}
\item $H^2_T(\mathbb{R}^d)$, the space of all predictable processes $\phi:[0,T]\times\Omega\mapsto \mathbb{R}^d$ such that $\E[\int_0^T |\phi_t|^2 \d t]<\infty$.
\end{itemize}

Assume that the generator $g:[0,T]\times \Omega\times \mathbb{R}^d\rightarrow \mathbb{R}$ satisfies requirements (i)-(ii) of Assumption \ref{a2}. For any fixed $t\in(0,T]$ and $X\in L^2(\mathcal{F}_t)$, consider the following type of BSDE:
\begin{equation}
\label{BSDE}
Y_s^X = X+\int_s^t g(r,Z_r^X)\d r - \int_s^t Z_r^X \d B_r,
\end{equation}
where $B=\{B_t\}_{t\in [0,T]}$ is a standard $\mathbb{F}$-Brownian motion under $\P_0$.

By \cite{PP90} there exists a unique pair $(Y^X,Z^X)\in H^2_T(\mathbb{R})\times H^2_T(\mathbb{R}^d)$ solving \eqref{BSDE}. In fact, $Y^X$ also satisfies $\E[\sup_{t\in[0,T]}|Y^X_t|^2]<\infty$.
 The \emph{conditional $g$-expectation of $X$} is then defined by
\begin{displaymath}
\mathcal{E}^g_{s,t}[X]:=Y_s^X, \quad s \in [0,T],\,\,t \in (0,T].
\end{displaymath}
If $t=T$ and $s=0$, for simplicity, we denote the \emph{$g$-expectation of $X$} by $\mathcal{E}^g[X]$. Furthermore, if $g$ also satisfies (iv) of Assumption \ref{a2}, it is easy to check that, for any $0\leq s\leq t_1\leq t_2\leq T$ and $X\in L^2(\mathcal{F}_{t_1})\subset L^2(\mathcal{F}_{t_2})$, we have
\begin{displaymath}
\mathcal{E}^g_{s,t_1}[X]=\mathcal{E}^g_{s,t_2}[X].
\end{displaymath}
Therefore, in this case, we may simplify the notation by omitting the terminal time and simply write $\mathcal{E}^g_s[X]$.

We now list some important (but standard) properties of $g$-expectation, without presenting their proof. For more details, we refer to the seminal papers \cite{CHMP} and \cite{P97}.
\begin{proposition}
\label{a21}
Suppose that the function $g$ satisfies (i)-(iii) of Assumption \ref{a2}. Then the conditional $g$-expectation is such that:
	\begin{itemize}
		\item[(1)] \textbf{Strict comparison}: if $X\leq Y$, then $\mathcal{E}^g_{t,T}[X]\leq \mathcal{E}_{t,T}^g[Y]$. Furthermore, if $\P(X<Y)>0$, then $\mathcal{E}^g_{t,T}[X] < \mathcal{E}_{t,T}^g[Y]$;
		\item[(2)] \textbf{Time-consistency}: for any $0\leq s\leq t\leq T$, $\mathcal{E}^g_{s,t}[\mathcal{E}^g_{t,T}[X]]=\mathcal{E}^g_{s,T}[X]$;
		\item[(3)] \textbf{Concavity}: $\mathcal{E}_{t,T}^g[\,\cdot\,]$ is concave; i.e., for any $X,Y\in L^2(\mathcal{F}_T)$ and $\lambda\in[0,1]$, we have $\mathcal{E}_{t,T}^g\big[\lambda X+(1-\lambda)Y\big]\geq \lambda\mathcal{E}_{t,T}^g[X]+(1-\lambda)\mathcal{E}_{t,T}^g[Y]$;
		\item[(4)] \textbf{Fatou's lemma}: Suppose that for any $n\in\mathbb{N}$, $\mathcal{E}^g[X_n]$ exists and $X_n\geq X$ (respectively, $X_n\leq X$), where $X\in L^2(\mathcal{F}_T)$. Then, we have
		$$
		\liminf_{n\rightarrow \infty}\mathcal{E}^g[X_n]\geq \mathcal{E}^g[\liminf_{n\rightarrow \infty} X_n] \quad (\text{respectively,}\,\, \limsup_{n\rightarrow \infty}\mathcal{E}^g[X_n]\leq \mathcal{E}^g[\limsup_{n\rightarrow \infty} X_n]).
		$$
	\end{itemize}
\end{proposition}

\begin{proposition}
\label{A2}
	Suppose that the function $g$ satisfies requirements (i)-(ii) and (iv) in Assumption \ref{a2}. Then the conditional $g$-expectation is such that:
	\begin{itemize}
		\item[(1)] \textbf{Translation invariance}: if $Z\in L^2(\mathcal{F}_t)$, then, for all $X\in L^2(\mathcal{F}_T)$, $\mathcal{E}^g_t[X+Z]=\mathcal{E}^g_t[X]+Z$;
		\item[(2)] \textbf{Local property}: for an event $A\in\mathcal{F}_t$, we have $\mathcal{E}^g_t[X\mathds{1}_A+Y\mathds{1}_{A^c}]=\mathcal{E}^g_t[X]\mathds{1}_A+\mathcal{E}^g_t[Y]\mathds{1}_{A^c}$;
		\item[(3)] \textbf{Constant preserving}: if $X\in L^2(\mathcal{F}_t)$, we have $\mathcal{E}^g_t[X]=X$.
	\end{itemize}
\end{proposition}



\section{}
\label{AppendixC}

\renewcommand{\theequation}{B-\arabic{equation}}

	In this Appendix we prove Theorem \ref{ii15}; that is, the necessity of the first-order conditions for optimality.
	
As in \cite{BR}, the proof will be divided into the following steps. First, we introduce a suitable linear problem and we characterize its solutions. Second, we show, by a perturbation method, that the optimal investment plan to \eqref{OC} also solves a linear optimization problem. We shall see that, because of our setting under Knightian uncertainty, some of the arguments in \cite{BR} needed a careful, technical, and not immediate extension (see, in particular, Lemmata \ref{l6} and \ref{l2} below).

\begin{lemma}
	\label{l1-nec}
		Let $\phi$ be a right-continuous and $\mathbb{F}$-adapted process, $\widehat{\P}$ a probability measure on $(\Omega,\mathcal{F})$, which is equivalent to $\P_0$, and denote by $\widehat{\E}$ the expectation under $\widehat{\P}$. Suppose that $I^*$ is optimal for the linear optimization problem
		\begin{equation}
		\label{linear}
		    \sup_{I \in\mathcal{A}'} \widehat{\E}\bigg[\int_0^T \phi_t \d I_t\bigg],
		\end{equation}
		where
		$$\mathcal{A}':=\Big\{I\in\mathcal{A}_g\,\big|\, C^I_t \textrm{ is square-integrable for any } t\in[0,T]\Big\}.$$
		
		Then, we have $\phi_t\leq 0$ for all $t$, $\P_0$-a.s. and
	\begin{equation}
		\label{linear2}
	\widehat{\E}\bigg[\int_0^T \phi_t \d I^{*}_t\bigg] = 0.
	\end{equation}
	\end{lemma}

    \begin{proof}
		Clearly, since $\mathcal{A}' \ni I\equiv 0$ is a priori suboptimal, one has
		$$\widehat{\E}\bigg[\int_0^T \phi_t \d I^{*}_t\bigg] \geq 0.$$
		
		To show the reverse inequality, suppose that there exists $t_o\in[0,T]$ such that $\esssup_{\omega \in \Omega}\phi_{t_o}(\omega)>0$. Then, since $\phi$ is adapted, there exists $\varepsilon>0$ and $Q\in\mathcal{F}_{t_o}$ with $\P_0(Q)>0$ such that $\phi_{t_o}\geq \varepsilon$ on $Q$. Noting that $\widehat{\P}$ is equivalent to $\P_0$, we have $\widehat{\P}(Q)>0$. Considering then the adapted, nondecreasing process
		$$\hat{I}_t:= I^{*}_t + \mathds{1}_{{Q}\times[t_o,T]}, \quad \hat{I}_{0^-}=0,$$
		one easily finds that $\hat{I}\in\mathcal{A}'$ and
		$$\widehat{\E}\bigg[\int_0^T \phi_t \d \hat{I}_t\bigg] \geq \widehat{\E}\bigg[\int_0^T \phi_t \d I^{*}_t\bigg] + \varepsilon \widehat{\P}(Q) > \widehat{\E}\bigg[\int_0^T \phi_t \d I^{*}_t\bigg],$$
		thus contradicting the optimality of $I^*$. Hence, $\phi_{t}\leq0$ for all $t$, $\P_0$-a.s.\ and \eqref{linear2} follows.
    \end{proof}

Before we establish that the optimal irreversible investment plan for problem \eqref{OC} also solves a linear problem like \eqref{linear}, we need the following technical result based on the continuity of $g$ and some estimates for BSDEs.
\begin{lemma}
\label{l6}
Suppose that the function $g$ satisfies Assumption \ref{a2} (i)-(iii) and (v) in Theorem \ref{ii15} and recall the set $D$ as defined in Proposition \ref{prop:A1}. Assume that $\{X_n\}_{n\in\mathbb{N}}$ and $X$ are random variables in $L^2(\mathcal{F}_T)$ such that
\begin{displaymath}
\lim_{n\rightarrow \infty}\E[|X_n-X|^2]=0.
\end{displaymath}

Then, there exists a family of $\mathbb{F}$-progressively measurable processes $\{\zeta^n\}_{n\in\mathbb{N}}\subset D$ and $\zeta\in D$ such that - denoting by $\P_n$ and $\P$ the probability measures with Girsanov kernel $\zeta^n$ and $\zeta$ (with respect to $\P_0$), respectively - one has that the density $\frac{\d \P_n}{d\P}$  is $p$-integrable for any $p\geq 1$, and $\frac{\d \P_n}{d\P} \to 1$ as $n\uparrow \infty$. Moreover,
\begin{displaymath}
\mathcal{E}^g[X_n]=\E^{\P_n}\left[X_n+\int_0^T f(s,\zeta^n_s)\d s\right],\  \mathcal{E}^g[X]=\E^\P\left[X+\int_0^T f(s,\zeta_s)\d s\right],
\end{displaymath}
where $f$ is the convex dual of $g$.
\end{lemma}

\begin{proof}
For $t\in[0,T]$, consider the following BSDEs under $\P_0$
\begin{align*}
Y_t&=X+\int_t^T g(s,Z_s)\d s-\int_t^T Z_s\d B_s,\\
Y_t^n&=X^n+\int_t^T g(s,Z_s^n)\d s-\int_t^T Z_s^n\d B_s.
\end{align*}
Let $\zeta$ and $\zeta^n$ be two progressively measurable processes such that, for any $s\in[0,T]$,
\begin{align}
\label{eqzeta}
f(s,\zeta_s)=g(s,Z_s)-\zeta_s Z_s, \ 	f(s,\zeta^n_s)=g(s,Z_s^n)-\zeta^n_s Z_s^n,
\end{align}
and define $\P$ and $\P_n$ as two probability measures with Girsanov kernel $\zeta$ and $\zeta^n$ (with respect to $\P_0$), respectively.
By the assumptions on $g$, we have $|\zeta_s|\leq \kappa$, $|\zeta^n_s|\leq \kappa$ for any $s\in[0,T]$.
Hence, by applying the Girsanov transformation, and using \eqref{eqzeta} we have
\begin{align*}
&Y_0=\mathcal{E}^{g}[X]=\E^{\P}\left[X+\int_0^T f(s,\zeta_s)\d s\right], \\ &Y_0^n=\mathcal{E}^{g}[X_n]=\E^{\P^n}\left[X_n+\int_0^T f(s,\zeta^n_s)\d s\right].
\end{align*}
Since standard estimates for solutions to BSDEs (cf.\ \cite{EPQ}) yield that
\begin{displaymath}
\E\left[\int_0^T|Z_s^n-Z_s|^2\d s\right]\leq C \E\big[|X-X_n|^2\big]\rightarrow 0, \textrm{ as } n\rightarrow \infty,
\end{displaymath}
we can choose a subsequence, still denoted by $\{Z^n\}$, such that $Z^n\rightarrow Z$, a.e., a.s. By Assumption (v), it follows that $\zeta^n\rightarrow \zeta$, a.e., a.s., which in turn implies that $\frac{d\P^n}{d\P}\rightarrow 1$.

Finally, the fact that $\frac{d\P^n}{d\P}$ is $p$-integrable for any $p\geq 1$ follows from the boundedness of $\zeta$ and $\zeta^n$.
\end{proof}

    \begin{lemma}
		\label{l2}
    	Let $I^\star\in\mathcal{A}_g$ be the optimal investment plan for \eqref{OC}. Recall $\mathcal{A}':=\{I\in\mathcal{A}_g\,|\, C^I_t \textrm{ is square-integrable for any } t\in[0,T]\}$. Suppose that for any $a>0$ and any $t\in[0,T]$ one has $\E\big[|\pi_c(X_t,a)|^2\big]<\infty$. Then, there exists some $\P^\star\in \mathcal{P}_g(I^\star)$ such that $I^\star$ is also optimal for the linear optimization problem
    	\begin{displaymath}
    		\sup_{I \in{\mathcal{A}}'}\E^{\P^\star}\bigg[\int_0^T e^{-rt}\phi^\star_t\d I_t\bigg],
    	\end{displaymath}
    where
    	\begin{displaymath}
    		\phi^\star_t:=\E^{\P^{\star}}_t\bigg[\int_t^T e^{-(r+\delta)(s-t)}\pi_c(X_s,C^\star_s)\d s\bigg]-1,
    	\end{displaymath}
    	and $C^\star$ is the production capacity associated to $I^\star$.
    \end{lemma}

    \begin{proof}
Let $I^\star$ be optimal for \eqref{OC}. By Lemma \ref{ii3} and Assumption \ref{a3}, we have $I^\star\in\mathcal{A}'$. For any constant $\varepsilon\in(0,1)$ and any $I\in\mathcal{A}'$, let $I^\varepsilon:=\varepsilon I+(1-\varepsilon)I^{\star}$. The capacities associated to $I$, $I^\varepsilon$ and $I^\star$   are denoted by $C$, $C^\varepsilon$, and $C^\star$, respectively. By \eqref{solC}, we have $C^\varepsilon=\varepsilon C+(1-\varepsilon) C^\star$ and $c+\int_0^t e^{\delta s}\d I^\star_s=e^{\delta t} C^\star_t$. It is easy to check that $I^\varepsilon\in\mathcal{A}'$.
	
	We claim (and prove later) that
     \begin{equation}
		\label{converge}
     \E\left[\bigg|\int_0^T e^{-rt}\big(\pi(X_t,C^\varepsilon_t)\d t-\d I^\varepsilon_t\big)-\int_0^T e^{-rt}\big(\pi(X_t,C^\star_t)\d t-\d I^\star_t\big)\bigg|^2\right]\rightarrow 0.
     \end{equation}
	Then, by Lemma \ref{l6}, there exist bounded and progressively measurable $\zeta^\varepsilon,\zeta \in D$ such that - letting $\P^{\varepsilon}$ and $\P$ the two probability measures with Girsanov kernels $\zeta^\varepsilon$ and $\zeta$ (with respect to $\P_0$), respectively - one has
\begin{align*}
      &\mathcal{E}^g\bigg[\int_0^T e^{-rt}\big(\pi(X_t,C^\varepsilon_t)\d t-\d I^\varepsilon_t\big)\bigg]=\E^{\P^\varepsilon}\bigg[\int_0^T e^{-rt}\big(\pi(X_t,C^\varepsilon_t)\d t-\d I^\varepsilon_t\big)+\int_0^T f(s,\zeta^\varepsilon_s)\d s\bigg],\\
      &\mathcal{E}^g\bigg[\int_0^T e^{-rt}\big(\pi(X_t,C^\star_t)\d t-\d I^\star_t\big)\bigg]=\E^{\P^\star}\bigg[\int_0^T e^{-rt}\big(\pi(X_t,C^\star_t)\d t-\d I^\star_t\big)+\int_0^T f(s,\zeta^\star_s)\d s\bigg].
      \end{align*}
      Since $I^\star$ is optimal for \eqref{OC}, $I^\varepsilon\in\mathcal{A}'\subset\mathcal{A}_g$ and $\pi$ is concave in its second argument, it is easy to check that
    	\begin{equation}\label{0}\begin{split}
    		0\geq &\frac{1}{\varepsilon}\bigg(\mathcal{E}^g\bigg[\int_0^T e^{-rt}\big(\pi(X_t,C^\varepsilon_t)\d t-\d I^\varepsilon_t\big)\bigg]-\mathcal{E}^g\bigg[\int_0^T e^{-rt}\big(\pi(X_t, C^\star_t)\d t-\d I^\star_t\big)\bigg]\bigg)\\
    \geq &\frac{1}{\varepsilon}\bigg(\E^{\P^\varepsilon}\bigg[\int_0^T e^{-rt}\big(\pi(X_t,C^\varepsilon_t)\d t-\d I^\varepsilon_t\big)\bigg]-\E^{\P^\varepsilon}\bigg[\int_0^T e^{-rt}\big(\pi(X_t,C^\star_t)\d t-\d I^\star_t\big)\bigg]\bigg)\\
    		\geq &\frac{1}{\varepsilon}\bigg(\E^{\P^\varepsilon}\bigg[\int_0^T e^{-rt} \pi_c(X_t,C^\varepsilon_t)(C_t^\varepsilon-C_t^\star)\d t\bigg]-\E^{\P^\varepsilon}\bigg[\int_0^T e^{-rt} \d (I^\varepsilon_t-I^\star_t)\bigg]\bigg)\\
    		=&\E^{\P^\varepsilon}\bigg[\int_0^T e^{-rt} \pi_c(X_t,C^\varepsilon_t)(C_t-C_t^\star)\d t\bigg]-\E^{\P^\varepsilon}\bigg[\int_0^T e^{-rt} \d(I_t-I^\star_t)\bigg]\\
    		=&\E^{\P^\star}\bigg[\int_0^T e^{-rt} \Phi_t^\varepsilon \d(I_t-I_t^\star)\bigg],
    	\end{split}\end{equation}
    	where
			$$\Phi_t^\varepsilon:=\frac{\d \P^\varepsilon}{\d \P^\star}\bigg(\int_t^T e^{-(r+\delta)(s-t)}\pi_c(X_s,C_s^\varepsilon)\d s-1\bigg).$$
     Let $$\Phi^{\star}_t:=\int_t^T e^{-(r+\delta)(s-t)}\pi_c(X_s,C_s^\star)\d s-1.$$

      Invoking again Lemma \ref{l6} thanks to \eqref{converge}, we have $\frac{\d \P^\varepsilon}{\d \P^\star}\rightarrow 1$ as $\varepsilon \downarrow 0$, and $\frac{\d \P^\varepsilon}{\d \P^\star}$ is $p$-integrable for any $p\geq1$. Hence, by Fatou's lemma
    	\begin{equation}
			\label{geq}
    		\liminf_{\varepsilon\rightarrow 0}\E^{\P^\star}\left[\int_0^T e^{-rt}\Phi^\varepsilon_t \d I_t\right]\geq \E^{\P^\star}\left[\int_0^T e^{-rt} \Phi^\star_t \d I_t\right].
    	\end{equation}
    	Also, if
    	\begin{equation}
			\label{leq}
    	\limsup_{\varepsilon\rightarrow 0}\E^{\P^\star}\left[\int_0^T e^{-rt}\Phi^\varepsilon_t \d I^*_t\right]\leq \E^{\P^\star}\left[\int_0^T e^{-rt} \Phi^\star_t \d I^\star_t\right],
    	\end{equation}
    then combining \eqref{0}-\eqref{leq}, we have
    	\begin{displaymath}
    		\E^{\P^\star}\left[\int_0^T e^{-rt} \Phi^\star_t \d I_t\right]\leq \E^{\P^\star}\left[\int_0^T e^{-rt} \Phi^\star_t \d I^\star_t\right].
    	\end{displaymath}
    	By Theorem (1.33) in \cite{J79}, we get the desired result.

	It thus only remains to prove \eqref{converge} and \eqref{leq}. This is accomplished below.
\vspace{0.25cm}

  \emph{Proof of Equation \eqref{converge}.}  By Equation \eqref{solC}, it is easy to check that, for any $t\in[0,T]$, $C^\star_t\geq e^{-\delta T}(c\vee I^\star_0)$. If the initial capacity is such that $c=0$, then the firm would invest immediately at time $0$; i.e., $I^\star_0>0$. It thus follows that, there exists some constant $c_0>0$, such that $C^\star_t\geq c_0$ and $C^\varepsilon_t\geq \frac{1}{2}C^\star_t\geq c_0$ for any $t\in[0,T]$, when $\varepsilon\leq \frac{1}{2}$. Noting that $\pi$ is concave in its second component, we have
   \begin{align*}
   &\pi(X_t,C^\varepsilon_t)-\pi(X_t,C^\star_t)\leq \pi_c(X_t,C^\star_t)(C^\varepsilon_t-C^\star_t)\leq \pi_c(X_t,c_0)|C^\varepsilon_t-C^\star_t|,\\
   &\pi(X_t,C^\star_t)-\pi(X_t,C^\varepsilon_t)\leq \pi_c(X_t,C^\varepsilon_t)(C^\star_t-C^\varepsilon_t)\leq \pi_c(X_t,c_0)|C^\varepsilon_t-C^\star_t|.
   \end{align*}
   By Equation \eqref{i8}, we obtain that
     \begin{align*}
     &\bigg|\int_0^T e^{-rt}\big(\pi(X_t,C^\varepsilon_t)\d t-\d I^\varepsilon_t\big)-\int_0^T e^{-rt}\big(\pi(X_t,C^\star_t)\d t-\d I^\star_t\big)\bigg|\\
     \leq &\bigg|\int_0^T e^{-rt}\big(\widetilde{\pi}(X_t,C^\varepsilon_t)-\widetilde{\pi}(X_t,C^\star_t)\big)\d t-e^{-rT}(C^\varepsilon_T-C^\star_T)\bigg|\\
     \leq &\int_0^T e^{-rt}\big(\pi_c(X_t,c_0)+r+\delta\big)|C^\varepsilon_t-C^\star_t|\d t+e^{-rT}|C^\varepsilon_T-C^\star_T|\\
      =&\varepsilon \int_0^T e^{-rt}\big(\pi_c(X_t,c_0)+r+\delta\big)|C_t-C^\star_t|\d t+\varepsilon e^{-rT}|C_T-C^\star_T|
     \end{align*}
     where $\widetilde{\pi}(x,c)=\pi(x,c)-(r+\delta)c$.
       Noting that $$\E[|\pi_c(X_t,c_0)|^2]<\infty,\ \E[|C_t|^2]<\infty,\ \E[|C^\star_t|^2]<\infty,$$ the proof is complete.
    \vspace{0.25cm}

  \emph{Proof of Equation \eqref{leq}.} Set $\eta^\varepsilon=\int_0^T e^{-rt}\Phi^\varepsilon_t \d I^\star_t$. To prove \eqref{leq}, it is sufficient to show that $\eta^\varepsilon$ is uniformly integrable, for any $0\leq \varepsilon\leq \frac{1}{2}$. Noting that $2C^\varepsilon \geq C^\star$ for any $0\leq \varepsilon\leq \frac{1}{2}$ and $\pi$ is concave, it is easy to check by Fubini's theorem that
    	\begin{align*}
    		 &\int_0^T e^{-rt} \bigg(\int_t^T e^{-(r+\delta)(s-t)}\pi_c(X_s,C_s^\varepsilon)\d s\bigg) \d I^\star_t\\
    =&\int_0^T \bigg(\int_0^s e^{-rt}e^{-(r+\delta)(s-t)}\pi_c(X_s,C_s^\varepsilon)\d I_t^\star\bigg) \d s\\
    		\leq &\int_0^T e^{-rs} \pi_c(X_s,C_s^\varepsilon) C^\star_s\d s\leq 2\int_0^T e^{-rs} \pi_c\big(X_s,\frac{1}{2}C_s^\star\big) \frac{1}{2}C^\star_s\d s\\
    		\leq &2\int_0^T e^{-rs}\pi\big(X_s,\frac{1}{2}C^\star_s\big)\d s\leq \int_0^T e^{-rs}\big(2\pi^\star(X_s,r,\delta)+(r+\delta)C_s^\star\big)\d s,
    	\end{align*}
    where we use Equation \eqref{solC} in the first inequality. Recalling Lemma \ref{ii3}, the optimal capacity  $C^\star$ satisfies $C^\star_t\leq \widehat{C}_t$, for any $t\in[0,T]$, where $\widehat{C}_t$ is defined in \eqref{Chat}. By arguments similar to those employed in \eqref{i8}, we have
    \begin{align*}
    \int_0^T e^{-rt}\d I_t^\star &\leq e^{-rT}C^\star_T+\int_0^T (r+\delta)e^{-rt}C^\star_t \d t \leq e^{-rT}\widehat{C}_T+\int_0^T (r+\delta)e^{-rt}\widehat{C}_t \d t\\
    & \leq M\big(1+\sup_{t\in[0,T]}c^\star(X_s,r,\delta)\big),
    \end{align*}
    where $M$ depends on $r,\delta,T,c$. Hence,
    $$|\eta^\varepsilon|\leq M\frac{\d \P^\varepsilon}{\d \P^\star}\bigg(\int_0^T e^{-rs}\pi^\star(X_s,r,\delta)\d s+1+\sup_{t\in[0,T]}c^\star(X_t,r,\delta)\bigg).$$
    Since $\frac{\d \P^\varepsilon}{\d \P^\star}$ is $p$-integrable for any $p\geq 1$, by Assumption \ref{a3} and H\"older's inequality one finds that $\eta^\varepsilon$ is indeed square-integrable, as claimed.
    \end{proof}
		
The claim of Theorem \ref{ii15} finally follows by combining the results of Lemma \ref{l1-nec} and Lemma \ref{l2} (upon identifying $\widehat{\P} = \P^\star$, $\phi_t=\phi^{\star}_t$ (which is adapted and right-continuous due to the right-continuity of $\mathbb{F}$), and $I^*=I^{\star}$ in Lemma \ref{l1-nec}).

	
	\section{}
\label{AppendixB}

\renewcommand{\theequation}{C-\arabic{equation}}

The proof of Proposition \ref{prop:v} will be obtained with the help of several auxiliary lemmata. For simplicity, set $Y_t:=(b-\frac{1}{2}\sigma^2)t+\sigma B_t$. Then the economic shock can be written as $X^x_t=x \exp(Y_t)$. Under $ \P^{-k}$, $Y$ is a L\'{e}vy process (a Brownian motion with drift, in fact). We denote its Laplace exponent under $\P^{-\kappa}$ by $\phi$; in particular, this is given by
$$\phi(\lambda)=\frac{1}{2}\sigma^2\lambda^2+(b-\frac{1}{2}\sigma^2-\sigma \kappa)\lambda, \quad \lambda \in \R.$$

Then $M^\lambda_t:=\exp(\lambda Y_t - \phi(\lambda) t)$ is a martingale under $\P^{-\kappa}$ and induces a new measure $\Q$ which is equivalent to  $\P^{-\kappa}$ on $(\Omega, \mathcal{F}_t)$, $t\geq0$ (note that in the following we shall always work under $\P^{-\kappa}$ here, not under our reference measure $\P_0$). Under $\Q$, $Y$ is also a L\'{e}vy process and we denote its Laplace exponent under $\Q$ by $\phi^{\Q}$.

\begin{lemma}
\label{lem1}
We have
$$\phi^{\Q}(\lambda)=\phi(\lambda+\alpha)-\phi(\alpha).$$
\end{lemma}

\begin{proof}
By definition of the Laplace exponent, we have
\begin{align*}
\exp(\phi^\Q(\lambda) ) &= \E^\Q \exp (\lambda Y_1 ) = \E^{-\kappa} \exp( \lambda Y_1 + \alpha Y_1 - \phi(\alpha) ) =\exp( \phi(\lambda+\alpha)-\phi(\alpha)).
\end{align*}
\end{proof}

Observing that
$$(X^x_t)^\alpha = x^\alpha e^{\alpha Y_t} = x^\alpha e^{\alpha Y_t- \phi(\alpha) t} e^{\phi(\alpha) t} =  x^\alpha M^\alpha_t  e^{\phi(\alpha) t},$$
the gross profit can be rewritten under $\P^{-\kappa}$ in the following way
\begin{align*}
\E^{-\kappa}\left[ \int_0^\infty  e^{-rt} \pi(X^x_t,C^{\star,-\kappa,c}_t)\d t \right]&=
\frac{x^\alpha}{1-\alpha} \E^{-\kappa}\left[\int_0^\infty e^{-rt + \alpha Y_t} (C^{\star,-\kappa,c}_t)^{1-\alpha} \d t \right]\\
&=
\frac{x^\alpha}{1-\alpha} \E^{-\kappa} \left[\int_0^\infty M^{\alpha}_t e^{-(r-\phi(\alpha))t} (C^{\star,-\kappa,c}_t)^{1-\alpha} \d t\right] \\
&=
\frac{x^\alpha}{1-\alpha} \E^{\Q}\left[ \int_0^\infty e^{- \tilde{r} t} (C^{\star,-\kappa,c}_t)^{1-\alpha} \d t\right],
\end{align*}
where $\tilde{r}:=r-\phi(\alpha)$. Now, if $\tilde{\tau}$ is an independent exponential random variable with parameter $\tilde{r}$, the gross profit can be finally rewritten as
\begin{align}
\label{eq:gross1}
\E^{-\kappa} \left[\int_0^\infty  e^{-rt} \pi(X^x_t,C^{\star,-\kappa,c}_t)\d t\right]  &=
\frac{x^\alpha}{(1-\alpha)\tilde{r}}  \E^{\Q} \big[(C^{\star,-\kappa,c}_{\tilde{\tau}})^{1-\alpha}\big] .
\end{align}

By Lemma 4.9 in \cite{BR}, one can show that the cost of the policy $I^{\star,-\kappa}$ can be written as
\begin{equation}
\label{EqnCost}
 \E^{-\kappa}\left[\int_0^\infty e^{-rt} dI^{\star,-\kappa}_t\right]=\E^{-\kappa} \left[\int_0^\infty e^{-rt} dC^{\star,-\kappa,c}_t\right] =  \E^{-\kappa} \big[C^{\star,-\kappa,c}_{ {\tau_r}}\big] - c,
 \end{equation}
where $\tau_r$ is an independent exponential time with parameter $r$. Summing up \eqref{eq:gross1} and \eqref{EqnCost}, the value function of our problem admits the representation
$$v(x,c)= \frac{x^\alpha}{(1-\alpha)\tilde{r}}  \E^{\Q} \big[(C^{\star,-\kappa,c}_{\tilde{\tau}})^{1-\alpha}\big]-\left(\E^{-\kappa} \big[C^{\star,-\kappa,c}_\tau\big] - c\right) .$$

To give explicit expressions for the two expectations above, we will use the fact that the running maximum of a Brownian motion stopped at an independent exponential time is exponentially distributed (see, e.g., \cite{handbook}). Let us recall two useful facts.
 \begin{lemma}\label{lem2}
 Denote by $Y^*$ the running maximum of $Y$.
 \begin{enumerate}
 \item $Y^*_{\tau_r}$ is exponentially distributed under $\P^{-\kappa}$ with parameter $\lambda>0$ that solves $\phi(\lambda)=r$.
 \item  $Y^*_{\tilde{\tau}}$  is exponentially distributed under $\Q$ with parameter $\mu=\lambda-\alpha$.
 \end{enumerate}
 \end{lemma}

\begin{proof}
We only need to prove the second assertion, as the first one can be found in \cite{handbook}. Note that $Y^*_{\tilde{\tau}}$  is exponentially distributed under $\Q$ with parameter $\mu>0$ that solves
$$ \phi^\Q(\mu)=\tilde{r}=r-\phi(\alpha) = \phi(\lambda) - \phi(\alpha).$$
Lemma \ref{lem1} yields the result.
\end{proof}

\begin{lemma}\label{lem3}
Let $Z$ be exponentially distributed with parameter $\nu>0$. Let $L\geq 1$ and $a<\nu$. Then
$$\E\big[\max\{ L, \exp(aZ)\}\big] = L + \frac{a}{\nu-a} L^{\frac{a-\nu}{a}}.$$
\end{lemma}

Bearing in mind the previous two lemmata, we can now prove Proposition \ref{prop:v}.

\begin{proof}[Proof of Proposition \ref{prop:v}]
We only need to derive the expression for $v(x,c)$ when $c>Kx$. With the help of Lemma \ref{lem3}, we obtain for the cost of our policy $I^{\star,-\kappa}$
\begin{align*}
 \E^{-\kappa}\big[ C^{\star,-\kappa,c}_{\tau_r}\big] - c &
= \E^{-\kappa}\big[ \max\{c, x K \exp(Y^*_{\tau_r})\}\big] - c\\
&
=  x K \E^{-\kappa} [\max\{\frac{c}{x K},  \exp(Y^*_{\tau_r})\}] - c\\
&=  x K  \left( \frac{c}{x K} + \frac{1}{\lambda-1} \left(\frac{c}{x K}\right)^{ 1-\lambda }\right)-c\\
&=
\frac{1}{\lambda-1}  (x K )^{ \lambda} c^{ 1-\lambda}.
\end{align*}

Next, let us compute the gross profit. It is given by
\begin{align*}
 \frac{x^\alpha}{(1-\alpha)\tilde{r}}  \E^{\Q} \big[\big(C^{\star,-\kappa,c}_{\tilde{\tau}}\big)^{1-\alpha}\big] &=   \frac{x^\alpha}{(1-\alpha)\tilde{r} }
 \E^{\Q} \big[\left( \max\{ c, xK  \exp(Y^*_{\tilde{\tau}})\}\right)^{1-\alpha} \big]\\
 &=
  \frac{x K^{1-\alpha}}{(1-\alpha)\tilde{r}} \E^{\Q} \left[\max\left\{
 \left( \frac{c}{xK}\right)^{1-\alpha}, \exp((1-\alpha) Y^*_{\tilde{\tau}})\right\}\right]\\
 &= \frac{x K^{1-\alpha}}{(1-\alpha)\tilde{r}}  \left(   \left( \frac{c}{xK}\right)^{1-\alpha} +  \frac{1-\alpha}{\mu+\alpha-1} \left( \frac{c}{xK}\right)^{(1-\alpha)\frac{1-\alpha-\mu}{1-\alpha}} \right) \\
&= \frac{x^\alpha c^{1-\alpha}}{(1-\alpha) \tilde{r} }
 +
\frac{1}{\lambda-1}
 \frac{x^\lambda K^\mu c^{1-\lambda}}{\tilde{r}},
 \end{align*}
where we have used Lemma \ref{lem3} and Lemma \ref{lem2} in the third and the fourth equality, respectively.
Note that the first term in the sum describes the profit that the agent obtains by never investing at all. The second term thus describes the profit from the additional investment.

 For the net profit, we thus obtain for any $c>Kx$
\begin{align*}
 v(x,c) &=
 \frac{x^\alpha c^{1-\alpha}}{(1-\alpha) \tilde{r} }
+ \frac{1}{\lambda-1}
 x^\lambda   c^{1-\lambda}
 \left( \frac{K^\mu}{\tilde{r}} - K^\lambda\right)\\
 &=
 \frac{x^\alpha c^{1-\alpha}}{(1-\alpha) \tilde{r} }
+ \frac{1}{\lambda-1}
 x^\lambda   c^{1-\lambda} K^\mu
 \left( \frac{1}{\tilde{r}} - K^\alpha\right).
 \end{align*}

Finally, the claimed regularity of $v$ follows by direct calculations exploiting the definition of $K$ (cf.\ \eqref{K-k-bis}).
 \end{proof}


\end{document}